\documentclass[final]{amsart}

\usepackage{showkeys}
\usepackage{amssymb}	
\usepackage{amsmath}
\usepackage{xcolor}
\usepackage{tikz}
\usepackage{tikz-cd}
\usepackage{tikzscale}
\usetikzlibrary{calc,math}
\usepackage{adjustbox}
\usepackage{cite}

\usepackage{mathrsfs}
\usepackage{mathtools}
\usepackage{amsthm}
\usepackage{verbatim}
\usepackage{subeqnarray}
\usepackage{graphicx}
\usepackage{cancel}
\usepackage{color}
\usepackage{float}
\usepackage{bm}
\usepackage{hyperref}
\usepackage{amsfonts}
\usepackage{amscd}
\usepackage{cases}
\usepackage{xcolor}
\usepackage[normalem]{ulem}

\definecolor{red}{rgb}{0.8,0,0}
\definecolor{darkorange}{rgb}{1,0.4,0}
\definecolor{lightorange}{rgb}{1,0.6, 0}
\definecolor{yellow}{rgb}{1,0.8, 0}

\newtheorem{theorem}{Theorem}
\newtheorem{remark}{Remark}

\newtheorem{definition}{Definition}

\newtheorem{lemma}{Lemma}

\newtheorem{assumption}{Assumption}
\newtheorem{example}{Example}

\newcommand{\charset}{X}
\newcommand{\charvec}{\bm\chi}

\newcommand\tr{\operatorname{tr}}
\newcommand\inc{\operatorname{inc}}
\newcommand\skw{\operatorname{skw}}
\newcommand\sskw{\operatorname{sskw}}
\newcommand\vskw{\operatorname{vskw}}
\newcommand\mskw{\operatorname{mskw}}

\newcommand\sym{\operatorname{sym}}
\newcommand\grad{\operatorname{grad}}
\newcommand\deff{\operatorname{def}}
\renewcommand\div{\operatorname{div}}
\renewcommand\ker{\mathcal{N}}

\newcommand\curl{\operatorname{curl}}
\newcommand\rot{\operatorname{rot}}

\newcommand\dev{\mathrm{dev}}

\newcommand \alt{\mathrm{Alt}}

\newcommand\hess{\operatorname{hess}}

\newcommand\spn{\operatorname{span}}
\newcommand\ran{\mathcal{R}}

\newcommand\K{\mathbb{K}}
\newcommand\M{\mathbb{M}}
\newcommand{\n}{\boldmath n}

\newcommand\dd{\,\mathrm{d}}
\newcommand\dx{\,\mathrm{d}x}
\renewcommand\S{{\mathbb S}}

\newcommand\R{\mathbb{R}}

\newcommand\V{{\mathbb{V}}}

\newcommand{\nodal}{\mathcal N}
\newcommand{\interval}{\mathcal I}
\newcommand{\interp}{\Pi}
\newcommand{\winterp}{\pi}
\newcommand{\wnodal}{\overline{\nodal}}
\newcommand{\opS}{\mathcal T}
\newcommand{\bs}{{\scriptscriptstyle \bullet}}

\newcommand{\br}{\mathbf r}
\newcommand{\vbr}{\vec{\mathbf r}}

\newcommand{\vp}{\vec{p}}



\author[F. Bonizzoni]{Francesca Bonizzoni}
\address{MOX - Department of Mathematics, Politecnico di Milano. Via Bonardi 9, 20133 Milano.}
\email{francesca.bonizzoni@polimi.it}
\thanks{FB acknowledges support from the HGS MathComp through the Distinguished Romberg Guest Professorship Program. Moreover, FB is member of the INdAM Research group GNCS and her work was part of a project that has received funding from the European Research Council ERC under the European Union's Horizon 2020 research and innovation program (Grant agreement No.~865751).}

\author[K. Hu]{Kaibo Hu}
\address{School of Mathematics, the University of Edinburgh, James Clerk Maxwell Building,
Peter Guthrie Tait Rd, Edinburgh EH9 3FD, UK.}
\email{kaibo.hu@ed.ac.uk}
\thanks{
KH was supported by a Hooke Research Fellowship and a Royal Society University Research Fellowship (URF$\backslash${\rm R}1$\backslash$221398)
}

\author[G. Kanschat]{Guido Kanschat}
\address{Interdisciplinary Center for Scientific Computing (IWR), Heidelberg University, Im Neuenheimer Feld 205, 69120 Heidelberg, Germany}
\email{kanschat@uni-heidelberg.de}
\thanks{GK was supported by the Deutsche Forschungsgemeinschaft (DFG, German Research Foundation) under Germany's Excellence Strategy EXC 2181/1 - 390900948 (the Heidelberg STRUCTURES Excellence Cluster)}

\author[D. Sap]{Duygu Sap}
\address{Mathematical Institute, University of Oxford, Andrew Wiles Building, Radcliffe Observatory Quarter, Oxford, OX2 6GG, UK.}
\email{duygu.sap@maths.ox.ac.uk}
\thanks{DS was supported by the Engineering and Physical Sciences Research Council (EPSRC) grants EP/S005072/1 and EP/R029423/1.}

\thanks{The authors would like to thank Espen Sande for helpful discussions.}

\begin{document}

\title[Tensor product BGG sequences: splines and FEM]{Discrete tensor product BGG sequences: splines and finite elements}

\maketitle

\begin{abstract}
In this paper, we provide a systematic discretization of the Bern\-stein-Gelfand-Gelfand (BGG) diagrams and complexes over cubical meshes in arbitrary dimension via the use of tensor-product structures of one-dimensional piecewise-polynomial spaces, such as spline and finite element spaces. We demonstrate the construction of the Hessian, the elasticity, and $\div\div$ complexes as examples for our construction. 
\end{abstract}

Differential complexes encode key algebraic structures in analysis and computation for partial differential equations. In addition to the de Rham complex with applications in electromagnetism, other complexes, such as the elasticity (Kr\"oner, Calabi, Riemannian deformation) complex, are drawing attention for problems in areas such as geometry, general relativity and continuum mechanics \cite{angoshtari2015differential,arnold2018finite,Arnold2006a,arnold2021complexes,pauly2016closed,pauly2018divdiv}. 
These complexes are special cases of the so-called Bernstein-Gelfand-Gelfand (BGG) sequences \cite{Arnold2006a,arnold2021complexes,vcap2022bgg,cap2009parabolic,vcap2001bernstein}.  Recently, a systematic study of the construction of BGG complexes and their properties was given in \cite{arnold2021complexes} with generalizations in \cite{vcap2022bgg}.

In the framework of the finite element exterior calculus (FEEC) \cite{arnold2018finite,Arnold.D;Falk.R;Winther.R.2006a,Arnold.D;Falk.R;Winther.R.2010a}, there has been a systematic discretization of the de Rham complex in any dimension for any form and polynomial degree. 
See also the finite element periodic table \cite{arnold2014periodic}. For example, conforming finite elements over cubical meshes were developed for the Stokes problem using de Rham complexes in $\mathbb{R}^n$ \cite{neilan2016stokes}.  In the context of isogeometric analysis (IGA) \cite{cottrell2009isogeometric}, spline de Rham complexes in 3D can be found in, e.g., \cite{buffa2011isogeometric}.

Since the seminal work on conforming finite elements on triangular meshes by Arnold and Winther \cite{arnold2002mixed}, there has been a lot of progress on discretizing the Hellinger-Reissner formulation of linear elasticity which involves the last two spaces of the elasticity complex \cite{arnold2018finite}. 
This was the first motivation and application of BGG complexes in numerical analysis \cite{Arnold2006a}, see also \cite{arnold2007mixed,eastwood2000complex}. Recently, there has been a surge of interests on discretizing elasticity and other special cases of BGG complexes \cite{arf2021structure,arnold2005rectangular, chen2020finite,chen2021finite,chen2021finite2,chen2020discrete,hu2021conforming,hu2021conforming2,hu2022new,sander2021conforming,christiansen2020discrete,christiansen2019finite}. See also \cite{hu2022oberwolfach} for a review.  
Moreover, we mention \cite{chen2022complexes,chen2023finite}, where a discretization for the entire BGG complex in 2D and 3D is proposed. There are also recent results on conforming simplicial finite elements for $(n-1)$-forms in $\mathbb{R}^n$ \cite{chen2021geometric}. Nevertheless, a systematic discretization of either the BGG complexes or the entire BGG machinery behind them in arbitrary dimension is still missing.
The underlying machinery can have important consequences beyond the final complexes (see, e.g., \cite{vcap2022bgg} for connections between twisted de Rham complexes and the Cosserat continua and \cite{christiansen2019poincare} for deriving Poincar\'e operators using the BGG machinery). 
This gap is a major issue to address in order to use the BGG construction for numerical computation. 

Restricted to the discrete level, the BGG machinery also provides a constructive tool for deriving discrete spaces by a diagram chasing. This idea has been developed at several places: a re-interpretation of the Arnold-Winther and Hu-Zhang elasticity elements \cite{Arnold2006a,christiansen2018nodal},   conforming finite elements for linearized curvature on 2D triangular meshes \cite{christiansen2019finite}, and a finite element elasticity complex on tetrahedral grids \cite{christiansen2020discrete}. These results were derived in a case-by-case manner. Special cases of BGG complexes  can be found in \cite{arf2021structure}.


In this paper, we provide a systematic discretization of the BGG diagrams and complexes based on tensor product construction in arbitrary dimension. In particular, the resulting spaces are tensor products of spaces in one space dimension (1D), and they are naturally defined over the unit cube $[0,1]^n$. However, their extension to rectangular domains and to finite elements with rectangular mesh cells is straightforward.
The tensor product construction is built upon two general assumptions on sequences in 1D. As examples of input spaces for 1D complexes, we consider both splines and finite elements. We verify that the assumptions of the construction in \cite{arnold2021complexes} on the continuous level hold for these tensor product spaces.
The resulting discrete BGG diagram is symmetric in the sense that the space of $(i, j)$-forms is isomorphic to the space of $(j, i)$-forms as they have the same coefficients.
The tensor product construction leads to bounded commuting (quasi-)interpolations for splines and finite elements, which are important for analyzing numerical schemes for PDEs.
Thus, it provides us with a powerful tool to obtain stable 
discretizations on cubical meshes. It also avoids the issue of intrinsic super-smoothness (c.f. \cite{floater2020characterization}) known from simplicial schemes.

The outline of the paper is as follows: In Section \ref{sec:review-bgg}, we review the BGG construction at the continuous level, and in Section \ref{sec:construction}, we demonstrate how we construct tensor-product BGG complexes for arbitrary form degree and space dimension.  In Sections 3 and 4, we derive spline and tensor product finite element BGG complexes and provide examples that include the elasticity complex, the $\div\div$ complex and the Hessian complex.  In Section \ref{sec:conclusion}, we provide a brief review of the main aspects of our construction and highlight our results. We also show why extending our construction to more delicate BGG complexes (e.g., the conformal complexes which require input of three de Rham complexes \cite{vcap2022bgg}) is nontrivial.

\section{BGG complexes of differential forms}
\label{sec:review-bgg}

\subsection{Review of the BGG construction}

We briefly review the main idea of the BGG construction. The BGG construction is a tool to derive more complexes from the de Rham complexes. In the setup in \cite{arnold2021complexes}, we consider the following {\it BGG diagram}:
\begin{equation}\label{Z-complex}
\begin{tikzcd}
0 \arrow{r} &Z^{0} \arrow{r}{D^{0}} &Z^{1} \arrow{r}{D^{1}} &\cdots \arrow{r}{D^{n-1}} & Z^{n}  \arrow{r}{} & 0\\
0 \arrow{r} &\tilde{Z}^{0}\arrow{r}{\tilde{D}^{0}} \arrow[ur, "S^{0}"]&\tilde{Z}^{1} \arrow{r}{\tilde{D}^{1}} \arrow[ur, "S^{1}"]&\cdots \arrow{r}{\tilde{D}^{n-1}}\arrow[ur, "S^{n-1}"] & \tilde{Z}^{n} \arrow{r}{} & 0.
 \end{tikzcd}
\end{equation}
Typically, each row of \eqref{Z-complex} is a scalar- or vector-valued de Rham complex, and the two rows are connected by some algebraic operators $S^{\bs}$ (in vector proxies, this can be, e.g., taking the skew-symmetric or the trace part of a matrix). We require that \eqref{Z-complex} is a commuting diagram, meaning that $D^{j+1}S^j=-S^{j+1}\tilde D^j$ for all $j=0,\ldots,n-1$ (the sign is not important, but just for technical reasons).
There is one index $J$ such that $S^{J}$ is bijective, and $S^{j}$ are injective for $j\leq J$ and surjective for $j\geq J$. From a BGG diagram, we can read out a {\it BGG sequence}. The recipe is that we start from the first row, restricted to $\ran(S)^{\perp}$ at each index (for example, if $S$ is taking the skew-symmetric part of a matrix, then $\ran(S)^{\perp}$ consists of symmetric matrices). Then at index $J$, we connect the two rows by constructing a new operator $\tilde{D}^J(S^{J})^{-1}D^J$. 

This is possible since $S^{J}$ is bijective and corresponds to a zig-zag on the diagram. Then we come to the second row. For the rest of the indices $j>J$, we restrict the domains to $\ker(S)$ in the second row. 
The final BGG complex is summarized as the following:
\begin{equation}\label{reduced-complex-1}
\begin{tikzcd}[row sep=small]
0 \arrow{r}&\Upsilon^{0}\arrow{r}{\mathscr{D}^{0}} &\Upsilon^{1} \arrow{r}{\mathscr{D}^{1}} &\cdots \arrow{r}{\mathscr{D}^{J-1}}&\Upsilon^{J}\\ \quad &\arrow{r}{\mathscr{D}^{J}} & \Upsilon^{J+1} \arrow{r}{\mathscr{D}^{j+1}} &\cdots\arrow{r}{\mathscr{D}^{n-1}} &\Upsilon^{n}\arrow{r}& 0,
 \end{tikzcd}
\end{equation}
with spaces
\begin{equation}\label{defUps}
\Upsilon^{j}:=
\begin{cases}
Z^0, & j=0,\\
 \ran ({S}^{j-1} )^{\perp}, & 1\leq j\leq J,\\
 \ker ({S}^{j} ), & J<j\leq n,
\end{cases}
\end{equation}
and operators
\begin{equation}\label{Dder}
\mathscr{D}^{j}=
\begin{cases}
  P_{\ran(S^j)^{\perp}} D^{j}, & 0\le j<J,\\
 \tilde{D}^{j}(S^{j})^{-1}D^{j},& j=J,\\
\tilde{D}^{j}, & J < j \le n,
\end{cases}
\end{equation}
where $P_{\ran(S^j)^{\perp}}$ denotes the orthogonal projection onto $\ran(S^j)^{\perp}$.

The main conclusion of this construction is that, under some conditions, the cohomology of the BGG complex \eqref{reduced-complex-1} is isomorphic to the cohomology of the input complexes in \eqref{Z-complex}.  For de Rham complexes, the cohomology is known with a broad class of function spaces \cite{costabel2010bogovskiui}. Various analytic properties follow from general arguments based on the fact that the cohomology being finite dimensional \cite{arnold2021complexes}. 

\begin{example}[BGG complex in 1D]
The simplest example is in one space dimension. We consider the following diagram with any real number $q$
 \begin{equation}\label{diagram-1D}
\begin{tikzcd}
0 \arrow{r} &H^{q}  \arrow{r}{\partial_{x}} &H^{q-1} \arrow{r}{} & 0\\
0 \arrow{r} &H^{q-1}\arrow{r}{\partial_{x}} \arrow[ur, "I"]&H^{q-2}  \arrow{r}{} & 0.
 \end{tikzcd}
\end{equation}
Here we only have one connecting operator, specifically, we have $S^0= I$ identity operator, which is obviously bijective (one may also extend the diagram by adding zero maps). All the conditions for the BGG diagrams hold trivially. From \eqref{diagram-1D} we derive the BGG complex 
\begin{equation}\label{BGG-complex1D}
\begin{tikzcd}
0 \arrow{r} &H^{q}  \arrow{r}{\partial_{x}^{2}} &H^{q-2} \arrow{r}{} & 0.
 \end{tikzcd}
\end{equation}
In this case, the BGG construction is just to connect the two first-order derivatives to get a second-order derivative. In this simple example it is also easy to verify our claim on the cohomology: the kernel of $\partial_{x}$ (cohomology at index 0) in each row is $\mathbb{R}$. In the BGG complex \eqref{BGG-complex1D}, the kernel of $\partial_{x}^{2}$ (cohomology at index 0) is isomorphic to $\mathbb{R}\oplus \mathbb{R}$.
\end{example}

\subsection{Application to alternating form-valued differential forms}

We follow the construction of alternating form-valued differential forms in~\cite{arnold2021complexes}.
For $i\ge 0$, let $\alt^i \R^n$ be the space of algebraic $i$-forms, that is, of alternating $i$-linear maps on $\mathbb{R}^{n}$.  We also set $\alt^{i, j}\R^n=\alt^{i}\R^n\otimes \alt^{j}\R^n$, the space of $\alt^{j}\R^n$-valued $i$-forms or, equivalently, the space of $(i+j)$-linear maps on $\mathbb{R}^{n}$ which are alternating in the first $i$ variables and also in the last $j$ variables. For the linking maps, we define the algebraic operators $s^{i, j}: \alt^{i, j} \R^n\to \alt^{i+1, j-1}\R^n$ 
\begin{multline}
\label{def:s}
s^{i, j}\mu (v_{0},\cdots,  v_{i})(w_{1}, \cdots, w_{j-1})
\\:=\sum_{l=0}^{i}(-1)^{l} \mu (v_{0},\cdots, \widehat{v_{l}}, \cdots,  v_{i})(v_{l}, w_{1}, \cdots, w_{j-1}), \\ \forall v_{0}, \cdots, v_{i}, w_{1}, \cdots, w_{j-1} \in \mathbb{R}^{n},
\end{multline}
where in each term of the alternating sum the \textit{hat} denotes the suppressed vector, that is, the vector we move from the first parenthesis to the second (see the appendix in~\cite{arnold2021complexes}).
Alternatively, we have the following expression of $s^{i,j}$ on a basis of alternating forms:
let $\sigma\in\Sigma(k,n)$ and $\tau\in\Sigma(m,n)$ be combinations of $k$ and $m$ elements of $\{1,\dots,n\}$, respectively. Then,
\begin{multline}
    s^{k, m}(dx^{\sigma_1}\wedge\cdots \wedge dx^{\sigma_k}\otimes dx^{\tau_{1}}\wedge \cdots \wedge dx^{\tau_{m}})
    \label{eq:s_iJ}
    \\=\sum_{l=1}^{m}(-1)^{l-1}dx^{\tau_l}\wedge dx^{\sigma_1}\wedge\cdots \wedge dx^{\sigma_k}\otimes dx^{\tau_1}\wedge \cdots \wedge \widehat{dx^{\tau_l}}\wedge\cdots\wedge dx^{\tau_m},
\end{multline}
where we move one factor from the second term in the tensor product to the first. For completeness, we include a detailed proof of equation~\eqref{eq:s_iJ} in Appendix \ref{sec:appendix}. From now on, we will omit $\mathbb{R}^n$ in the notation when confusion is unlikely.
We also write $S^{i,j}=I\otimes s^{i,j}:H^q\otimes\alt^{i, j}\to H^q\otimes\alt^{i+1, j-1}$ for any Sobolev order $q$. 
We use $H^{q}\Lambda^{i}$ as another notation for $H^{q}\otimes \alt^{i}$.
We have the exterior derivative, $d^{i}: H^{q}\otimes \alt^{i}\to H^{q-1}\otimes \alt^{i+1}$. Tensorizing with $\alt^{j}$ then gives $d^{i}: H^{q}\otimes \alt^{i,j}\to H^{q-1}\otimes \alt^{i+1,j}$. With these definitions, we may write down the diagram generalizing \eqref{diagram-1D} 
to $n$ dimensions: 
\begin{equation}\label{diagram-nD}
  \adjustbox{scale=0.9,center}{%
\begin{tikzcd}[column sep=tiny]
0 \arrow{r} & H^{q}\otimes\alt^{0,0}  \arrow{r}{d} &H^{q-1}\otimes\alt^{1,0}  \arrow{r}{d} & \cdots \arrow{r}{d} & H^{q-n}\otimes \alt^{n,0} \arrow{r}{} & 0\\
0 \arrow{r} & H^{q-1}\otimes\alt^{0,1}  \arrow{r}{d} \arrow[ur, "S^{0,1}"] &H^{q-2}\otimes\alt^{1,1}  \arrow{r}{d}  \arrow[ur, "S^{1,1}"] & \cdots \arrow{r}{d}  \arrow[ur, "S^{n-1,1}"] & H^{q-n-1}\otimes \alt^{n,1} \arrow{r}{} & 0\\[-15pt]
 & \vdots & \vdots & {} & \vdots & {} \\[-15pt]
0 \arrow{r} & H^{q-n+1}\otimes\alt^{0,n-1}  \arrow{r}{d} &H^{q-n}\otimes\alt^{1,n-1}  \arrow{r}{d} & \cdots \arrow{r}{d} & H^{q-2n+1}\otimes \alt^{n,n-1} \arrow{r}{} & 0\\
0 \arrow{r} & H^{q-n}\otimes\alt^{0,n}  \arrow{r}{d} \arrow[ur, "S^{0,n}"] &H^{q-n-1}\otimes\alt^{1,n}  \arrow{r}{d}  \arrow[ur, "S^{1,n}"] & \cdots \arrow{r}{d}  \arrow[ur, "S^{n-1,n}"] & H^{q-2n}\otimes \alt^{n,n} \arrow{r}{} & 0.
\end{tikzcd}}
\end{equation}

\begin{example}[BGG complexes in 3D]
The BGG diagram~\eqref{diagram-nD} has the following vector/matrix proxies for $n=3$, where the operators between proxies are defined in Appendix~\ref{sec:appendix-proxies}:
 \begin{equation}\label{diagram-4rows}
   \adjustbox{scale=0.9,center}{%
\begin{tikzcd}
0 \arrow{r} &H^{q}\otimes \mathbb{R}  \arrow{r}{\grad} &H^{q-1}\otimes \mathbb{V} \arrow{r}{\curl} &H^{q-2}\otimes \mathbb{V} \arrow{r}{\div} & H^{q-3}\otimes \mathbb{R} \arrow{r}{} & 0\\
0 \arrow{r}&H^{q-1}\otimes \mathbb{V}\arrow{r}{\grad} \arrow[ur, "I"]&H^{q-2}\otimes \mathbb{M}  \arrow{r}{\curl} \arrow[ur, "2\vskw"]&H^{q-3}\otimes \mathbb{M} \arrow{r}{\div}\arrow[ur, "\tr"] & H^{q-4}\otimes \mathbb{V} \arrow{r}{} & 0\\
0 \arrow{r} &H^{q-2}\otimes \mathbb{V}\arrow{r}{\grad} \arrow[ur, "-\mskw"]&H^{q-3}\otimes \mathbb{M}  \arrow{r}{\curl} \arrow[ur, "\opS"]&H^{q-4}\otimes \mathbb{M} \arrow{r}{\div}\arrow[ur, "2\vskw"] & H^{q-5}\otimes \mathbb{V} \arrow{r}{} & 0\\
0 \arrow{r} &H^{q-3}\otimes \mathbb{R}\arrow{r}{\grad} \arrow[ur, "\iota"]&H^{q-4}\otimes \mathbb{V}  \arrow{r}{\curl} \arrow[ur, "-\mskw"]&H^{q-5}\otimes \mathbb{V} \arrow{r}{\div}\arrow[ur, "I"] & H^{q-6}\otimes \mathbb{R} \arrow{r}{} & 0.
 \end{tikzcd}}
\end{equation}
Here $\mathbb V:=\mathbb R^n$ denotes vectors and $\mathbb M$ is the space of all $n\times n$-matrices. Let further $\mathbb S$, $\mathbb K$, and $\mathbb T$ be the subspaces of matrices that are symmetric, skew-symmetric and trace-free, respectively.
Following the BGG recipe, from the first two rows of \eqref{diagram-4rows} we obtain the {Hessian complex} 
\begin{equation}\label{grad-grad0}
\begin{tikzcd}[cramped]
0\arrow{r}  & H^{q}\otimes \mathbb{R} \arrow{r}{\hess} & H^{q-2}\otimes \mathbb{S} \arrow{r}{\curl} & H^{q-3}\otimes \mathbb{T} \arrow{r}{\div} & H^{q-4}\otimes \mathbb{V} \arrow{r} & 0,
\end{tikzcd}
\end{equation}
where $\hess:=\grad\grad$.
 From the second and third rows of \eqref{diagram-4rows} we obtain the {elasticity complex}
 \begin{equation}\label{sequence:hs}
   \adjustbox{scale=0.9,center}{%
\begin{tikzcd}
0\arrow{r} & H^{q-1}\otimes \mathbb{V} \arrow{r}{{\deff}} & H^{q-2}\otimes \mathbb{S} \arrow{r}{\inc} & H^{q-4}\otimes \mathbb{S} \arrow{r}{\div} & H^{q-5}\otimes \mathbb{V} \arrow{r} & 0.
\end{tikzcd}}
\end{equation}
Here $\deff:=\sym\grad$ is the linearized deformation (symmetric part of gradient) and $\inc = \curl \opS^{-1}\curl$ leads to the linearized Einstein tensor. 
Finally, the last two rows of \eqref{diagram-4rows} yield the divdiv complex
 \begin{equation}\label{div-div0}
   \adjustbox{scale=0.9,center}{%
\begin{tikzcd}[row sep=large, column sep = large]
0 \to H^{q-2}\otimes \mathbb{V}  \arrow{r}{\dev\grad} & H^{q-3}\otimes \mathbb{T}  \arrow{r}{\sym\curl} & H^{q-4}\otimes \mathbb{S}  \arrow{r}{\div\div} & H^{q-6}\otimes \mathbb{V} \to 0.
\end{tikzcd}}
\end{equation}
\end{example}

\begin{remark}
On the continuous level, the BGG construction \cite{arnold2021complexes} starts with \eqref{diagram-4rows} and derives BGG complexes~\eqref{grad-grad0}-\eqref{div-div0} consisting of Sobolev $H^q$ functions with certain $q$.
Since $q$ is decreasing with every derivative operator, this requires excessive regularity on the left. This can be avoided introducing the domain complex with spaces $H(D)=\{u\in L^2: Du \in L^2\}$, where $D$ is any of the derivative operators in~\eqref{grad-grad0}-\eqref{div-div0}.

On the discrete level, we will fit in finite element or spline spaces which are compatible with the algebraic structures in \eqref{diagram-4rows} (in the sense that the horizontal and diagonal operators map one space to another). However, we are not constrained by the Sobolev regularity in \eqref{diagram-4rows} as the resulting BGG complexes~\eqref{grad-grad0}-\eqref{div-div0} usually have higher regularity than one needs. In the end, we will arrive at conforming discretization for complexes with $H(D)$ spaces with slightly higher regularity. For example, in the case with lowest regularity, our elasticity complex starts with an $H^1(\curl)$-conforming space, rather than $H^1=H(\deff)$, see for instance $V^{0, 1}=V^{1, 0}$ in \eqref{diagram-V10}. Nevertheless, the $H^1(\curl)$ conformity is still weaker than $H^4$ as required by \eqref{sequence:hs} if all the spaces there are to be at least $L^2$. 
\end{remark}

\begin{remark}\label{isomorphic-cx}
In the construction of simplicial finite elements or splines, the isomorphism between the $\grad$-$\rot$ version of complexes and the $\curl$-$\div$ version is obtained by swapping the tangent and normal directions. In the tensor product construction, this is obtained via changing parametric directions accordingly. 
\end{remark}
 
\begin{example}[BGG Complex in 2D]
Let $\sskw=\mskw^{-1}\circ \skw: \mathbb{M}\to \mathbb{R}$ be the map taking the skew part of a matrix and identifying it with a scalar (see \cite{arnold2021complexes}).

A 2D version of the diagram \eqref{diagram-4rows} is
\begin{equation}\label{diagram-3rows2D}
\begin{tikzcd}
0 \arrow{r}  &H^{q}\otimes \mathbb{R} \arrow{r}{\grad} &H^{q-1}\otimes \mathbb{V} \arrow{r}{\rot} & H^{q-2}\otimes \mathbb{R} \arrow{r}{} & 0\\
0 \arrow{r} & H^{q-1}\otimes \mathbb{V} \arrow{r}{\grad}\arrow[ur, "I"] &H^{q-2}\otimes \mathbb{M} \arrow{r}{\rot}\arrow[ur, "-2\sskw"] & H^{q-3}\otimes \mathbb{V} \arrow{r}{} & 0\\
0 \arrow{r}  &H^{q-2}\otimes \mathbb{R} \arrow{r}{\grad}\arrow[ur, "\mskw"] &H^{q-3}\otimes \mathbb{V} \arrow{r}{\rot}\arrow[ur, "I"] & H^{q-4}\otimes \mathbb{R} \arrow{r}{} & 0.
 \end{tikzcd} 
\end{equation}
\end{example}

Although the injectivity/surjectivity conditions are not necessary for running the BGG machinery \cite{vcap2022bgg}, we will stick to these conditions in our construction on the discrete level, for simplicity. To summarize, on the discrete level, we seek discrete versions of the BGG diagrams such that the injectivity/surjectivity conditions are preserved. Then the commutativity $DS=-S\tilde{D}$ is trivially following the results on the continuous level. To this end, we need discrete de Rham complexes for the two rows in the construction, respectively. In fact, this is rather easy to see in 1D. To discretize \eqref{diagram-1D}, we want to construct discrete spaces $V_{h}^{i, j}$ that fit in the following diagram
 \begin{equation}\label{diagram-1D-h}
\begin{tikzcd}
0 \arrow{r} &V^{0, 0}_{h}  \arrow{r}{\partial_{x}} &V^{1, 0}_{h}\arrow{r}{} & 0\\
0 \arrow{r} &V^{0, 1}_{h}\arrow{r}{\partial_{x}} \arrow[ur, "I"]&V^{1, 1}_{h}  \arrow{r}{} & 0.
 \end{tikzcd}
\end{equation}
The connecting map (the identity map) makes sense if $V^{1, 0}_{h}\cong V^{0, 1}_{h}$. This means that the first row has higher regularity than the second. This is a general pattern in the BGG construction, but becomes more complicated in higher dimensions.

\section{Tensor product construction}
\label{sec:construction}

In this section, we present a general construction. The idea is that we start with a diagram in 1D, and extend it to $\mathbb R^n$ by tensor product. This process is mostly algebraic and does not depend on a particular construction in 1D. Thus, we start with a fairly abstract assumption.

\subsection{A discrete BGG complex in 1D}

The general construction is based on the following assumption.
\begin{assumption}
\label{assump:1Dcomplex}
Let $\mathcal{I}:=[0, 1]$, and $L^2\Lambda^i(\mathcal{I}):=L^2(\mathcal{I})\otimes \alt^i\mathbb{R}$.
For an abstract smoothness parameter $\br$ and integer $p\geq 1$,  let $\mathcal{S}_{\br}^{p}$ be a finite dimensional subspace of $L^2(\mathcal{I})$ and thus $\mathcal{S}_{\br}^{p}\Lambda^i(\mathcal{I}):=\mathcal{S}_{\br}^{p}\otimes\alt^i\mathbb{R}$ be a finite dimensional subspace of $L^2\Lambda^i(\mathcal{I})$. Furthermore, we assume that the sequence 
\begin{equation}
\label{sequence:assumption} 
\begin{tikzcd}
0 \arrow{r}&\mathcal{S}_{\br}^{p}\Lambda^0(\mathcal{I}) \arrow{r}{d^{0}}& \mathcal{S}_{\br-1}^{p-1}\Lambda^1 (\mathcal{I})\arrow{r}{}&0
\end{tikzcd}
\end{equation} 
is a complex and $d^0$ is onto.
\end{assumption}

As we shall see later in specific examples, $\mathcal{S}_{\br}^{p}(\mathcal{I})$ may be a spline space on $\mathcal{I}$ of degree $p$ with regularity vector $\br$ (see Section~\ref{sec:examples_splines}) or a finite element space of degree $p$ and interelement continuity $r$ (see Section~\ref{sec:examples_fe}). The decreasing indices in \eqref{sequence:assumption} reflect the fact that $d^0$ is a first-order differential operator. For splines, since $\br$ is a vector, $\br\ge 0$ means every component of $\br$ is non-negative. Similarly, in expressions such as $\br-a$ where $a\in \mathbb{R}$, $a$ is subtracted from each component, see for instance~\cite[p.821]{buffa2011isogeometric}. 

When there is no danger of confusion, we omit $\mathcal{I}$ in the notation. 
Now we tensor the spaces in \eqref{sequence:assumption} with alternating forms and obtain with the above notation
\begin{equation}\label{def:Sp}
  \mathscr{S}_{\br}^p\Lambda^{i, j}(\mathcal{I})
  :=\mathcal{S}_{\br-i-j}^{p-i-j}\Lambda^{i, j}(\mathcal{I})
  :=\mathcal{S}_{\br-i-j}^{p-i-j}\otimes \alt^{i}\mathbb{R}\otimes \alt^j\mathbb{R},
\end{equation}
as the space of 1D alternating $i, j$-forms with coefficients in $\mathcal{S}_{\br-i-j}^{p-i-j}$. 
 Following Assumption~\ref{assump:1Dcomplex}, the following sequences are complexes for $j=0,1$ and $d^0$ is onto:
\begin{equation}\label{1D-ijcomplex}
\begin{tikzcd}
0 \arrow{r}&\mathcal{S}_{\br-j}^{p-j}\Lambda^{0, j}(\mathcal{I}) \arrow{r}{d^{0}}& \mathcal{S}_{\br-1-j}^{p-1-j}\Lambda^{1, j}(\mathcal{I})\arrow{r}{}&0.
\end{tikzcd}
\end{equation} 
Hence, we obtain a 1D BGG diagram that satisfies the assumptions in Section \ref{sec:review-bgg}:
 \begin{equation}\label{1D:BGG-diagram}
\begin{tikzcd}
0 \arrow{r} &\mathcal{S}_{\br}^{p}\Lambda^{0, 0} (\mathcal{I}) \arrow{r}{d^{0}} &\mathcal{S}_{\br-1}^{p-1}\Lambda^{1,0} (\mathcal{I})\arrow{r}{} & 0\\
0 \arrow{r} &\mathcal{S}_{\br-1}^{p-1}\Lambda^{0, 1}(\mathcal{I})\arrow{r}{d^{0}} \arrow[ur, "S^{0, 1}"]&\mathcal{S}_{\br-2}^{p-2}\Lambda^{1, 1}(\mathcal{I}) \arrow{r}{} & 0.
 \end{tikzcd}
\end{equation}

\begin{remark}
\label{rem:s}
In a vector proxy with canonical bases, $S^{0, 1}$ boils down to the identity operator. In 1D, there is only one $s^{k,m}$ operator according to~\eqref{def:s}, namely $s^{0,1}\colon \alt^{0,1}\to \alt^{1,0}$. Nevertheless, it will be useful for the construction of tensor products, namely Lemma~\ref{lem:s-tensorpro}, to define $s^{k,m}\equiv 0$ for all other $k,m\in\{0,1\}$.
\end{remark}

\subsection{Tensor product spaces and the exterior derivative}\label{sec:tsp_extderivs}

Let $$\sigma = (\sigma_1,\dots,\sigma_k)\in\Sigma(k,n)$$ be a combination of $k$ numbers from 
$\{1,\cdots,n\}$ such that $1\le \sigma_1<\dots<\sigma_k\le n$. Define the set of characteristic vectors
\begin{gather}
\charset_{k}:=\left\{
    \bm{\chi}=(\chi_{1}, \cdots, \chi_{n})\in \{0, 1\}^{n} \middle| \sum_{j=1}^{n} \chi_{j}=k
    \right\}.
\end{gather}
Then, the characteristic vector of a combination $\sigma\in\Sigma(k,n)$ is the vector $\charvec(\sigma)\in\charset_{k}$, such that
\begin{gather*}
    \chi_i(\sigma)=
    \begin{cases}
    1, & i\in\sigma,\\
    0, & i \not\in\sigma.
    \end{cases}
\end{gather*}
Using combinations, we can define a basis for $\alt^k\R^n$ consisting of elements
\begin{gather}
\label{eq:form-basis}
\begin{aligned}
    dx^\sigma &= dx^{\sigma_1} \wedge \dots \wedge dx^{\sigma_k}
    &\sigma&\in\Sigma(k,n),\\
    &= (dx^1)^{s_1} \wedge\dots\wedge (dx^n)^{s_n}
    \qquad
    &\bm s&=\charvec(\sigma)\in\charset_{k}.
\end{aligned}   
\end{gather}
Here, we define $(dx^i)^1 = dx^i$ and $(dx^i)^0 = 1$. Note that these are two different notations for the same form and we will use both of them below for convenience. 
For $\charvec\in\charset_{n}$ we introduce the notation 
\begin{equation}
    \label{eq:Iell}
    |\charvec|_{m}:=\sum_{l=1}^{m} \chi_{l},
\end{equation} 
with the convention $|\charvec|_{0}=0$.

In $\mathbb R^{n}$, define the unit hypercube $\mathcal{I}^n: =[0, 1]^n$. We define the following Sobolev spaces of alternating forms:
\begin{alignat*}2
L^{2}\Lambda^{i, j}(\mathcal{I}^n)
&:=L^2\Lambda^i(\mathcal{I}^n)\otimes \alt^{j}\mathbb R^{n}
&&=L^{2}(\mathcal{I}^n)\otimes \alt^{i,j}\mathbb R^{n},
\\
H^{q}\Lambda^{i, j}(\mathcal{I}^n)
&:=H^q\Lambda^i(\mathcal{I}^n)\otimes \alt^{j}\mathbb R^{n}
&&=H^{q}(\mathcal{I}^n)\otimes \alt^{i,j}\mathbb R^{n}. 
\end{alignat*}

Let $\sigma\in\Sigma(i,n)$ and $\tau\in\Sigma(j,n)$ be combinations with characteristic vectors 
$\bm s=\charvec(\sigma)$ and $\bm t=\charvec(\tau)$,
 respectively. Let
$\omega_k = \alpha_k (dx)^{s_k}\otimes (dx)^{t_k} \in L^2\Lambda^{s_k,t_k}$ for $k=1,\dots,n$ be one-dimensional differential forms with an 
$L^2$ coefficient $\alpha_k$.
Then, the $n$-dimensional tensor product is defined as
\begin{align*}
   & \omega_1\otimes\dots\otimes\omega_n\\
    &\quad = (\alpha_1\otimes \dots \otimes \alpha_n)
    (dx^1)^{s_1}\wedge\dots\wedge (dx^n)^{s_n}
    \otimes (dx^1)^{t_1}\wedge\dots\wedge (dx^n)^{t_n}\\
    &\quad = (\alpha_1\otimes \dots \otimes \alpha_n) dx^\sigma\otimes dx^\tau\\
    &\quad = (\alpha_1\otimes \dots \otimes \alpha_n)
    [(dx^1)^{s_1} \otimes (dx^1)^{t_1}] \wedge \dots \wedge
    [(dx^n)^{s_n} \otimes (dx^n)^{t_n}].
\end{align*}
Here we used different combinations of wedge and tensor products to express the same object. The different notations are used in various contexts below and should be clear from these identities.
We define the tensor product spaces:
\begin{align}
\nonumber
L^{2}\Lambda^{k, l}_{\otimes n}(\mathcal I^{n})
&\coloneqq\hspace{-0.5cm}
\bigoplus_{\substack{(s_{i}, \cdots, s_{n})\in \charset_{k}\\ (t_{1}, \cdots, t_{n})\in \charset_{l}}} L^{2}\Lambda^{s_{1}, t_{1}}(\mathcal{I})\otimes L^{2}\Lambda^{s_{2}, t_{2}}(\mathcal{I})\otimes \cdots\otimes L^{2}\Lambda^{s_{n}, t_{n}}(\mathcal{I})\\
\label{def:L2}
&=\hspace{-0.5cm}
\bigoplus_{\substack{(s_{i}, \cdots, s_{n})\in \charset_{k}\\ (t_{1}, \cdots, t_{n})\in \charset_{l}}} L^{2}(\mathcal{I}^n)\otimes \alt^{s_{1}, t_{1}}\otimes \cdots \otimes \alt^{s_{n}, t_{n}},
\end{align}
\begin{align}
\nonumber
\mathscr{S}_{\vbr}^{\vp}\Lambda^{k, l}
&\coloneqq\hspace{-0.5cm}
\bigoplus_{\substack{(s_{i}, \cdots, s_{n})\in \charset_{k}\\ (t_{1}, \cdots, t_{n})\in \charset_{l}}}\mathscr{S}^{p_{1}}_{\br_{1}}\Lambda^{s_{1}, t_{1}}(\mathcal I)\otimes \cdots \otimes \mathscr{S}^{p_{n}}_{\br_{n}}\Lambda^{s_{n}, t_{n}}(\mathcal I)
\\
\nonumber
&=\hspace{-0.5cm}
\bigoplus_{\substack{(s_{i}, \cdots, s_{n})\in \charset_{k}\\ (t_{1}, \cdots, t_{n})\in \charset_{l}}}\mathcal{S}^{p_{1}-s_1-t_1}_{\br_{1}-s_1-t_1}\Lambda^{s_{1}, t_{1}}(\mathcal I)\otimes \cdots \otimes \mathcal{S}^{p_{n}-s_n-t_n}_{\br_{n}-s_n-t_n}\Lambda^{s_{n}, t_{n}}(\mathcal I)
\\\label{def:Sp2}
&=\hspace{-0.5cm}
\bigoplus_{\substack{(s_{i}, \cdots, s_{n})\in \charset_{k}\\ (t_{1}, \cdots, t_{n})\in \charset_{l}}}
\left(\mathcal{S}^{p_{1}-s_1-t_1}_{\br_{1}-s_1-t_1}\otimes \cdots \otimes \mathcal{S}^{p_{n}-s_n-t_n}_{\br_{n}-s_n-t_n}\right) \otimes \alt^{s_{1}, t_{1}}\otimes \cdots \otimes \alt^{s_{n}, t_{n}}.
\end{align}
Note that, the spaces $L^{2}\Lambda^{k, l}(\mathcal{I}^n)$ and $L^{2}\Lambda^{k, l}_{\otimes n}(\mathcal I^{n})$ coincide, see~\cite[Example 3.7]{Hackbusch14}. 
Equations~\eqref{def:L2} and~\eqref{def:Sp2} give a characterization that separates the coefficients and the alternating form basis. 

$\mathscr{S}_{\vbr}^{\vp}\Lambda^{k, l}$ is the space of $(k+l)$-linear maps on $\mathbb R^n$ which alternate in the first $k$ and last $l$ variables. It is a direct sum corresponding to the components of the vector proxies.
Thus, each component of the direct sum corresponds to a particular basis form $dx^\sigma\otimes dx^\tau\in \alt^{k,l}$ according to equation~\eqref{eq:form-basis} together with its coefficient space.
The coefficient space is obtained from a ``root'' space $\mathcal{S}^{p_1}_{\br_1}\otimes\dots\otimes\mathcal{S}^{p_n}_{\br_n}$
by reducing the order and regularity by one in each direction $i$ where $i$ is contained either in $\sigma$ or in $\tau$. Accordingly, they are reduced by two if $i$ is contained in $\sigma$ and in $\tau$.
This means in order for the  discrete spaces to be compatible with the $S$ operators, i.e., for $S$ to map the discrete space $\mathscr{S}_{\vbr}^{\vp}\Lambda^{k, l}$ to the right one $\mathscr{S}_{\vbr}^{\vp}\Lambda^{k+1, l-1}$, the polynomial degrees in direction $i$ should always involve the sum $(s_i+t_i)$. Thus, we have polynomial spaces of the form $\mathcal{S}^{p_{i}-s_i-t_i}_{\br_{i}-s_i-t_i}$.\\
This structure of the spaces induces symmetry in the discretization of the BGG diagram \eqref{diagram-nD}, such that the $(k, l)$-th space is isomorphic to the $(l, k)$-th space, as the polynomial coefficients in the definition of $\mathscr{S}_{\vbr}^{\vp}\Lambda^{k, l}$ are invariant when we switch $k$ and $l$.
For example, the first row of \eqref{diagram-nD} is a standard tensor product de Rham complex \cite{Arnold-Boffi-Bonizzoni,christiansen2009foundations}, and {\it so is the first column}. As a more specific example, the 3D elasticity complex starts with a (0, 1)-form (the first space in the second row of \eqref{diagram-4rows}), indicating that our discrete elasticity complex starts from a N\'ed\'elec space. See Section \ref{sec:examples_fe} for more details.

The exterior derivatives $d^{k}: \mathcal{S}_{\br}^{p}\Lambda^{k, l}(\mathcal{I})\to \mathcal{S}_{\br-1}^{p-1}\Lambda^{k+1, l}(\mathcal{I})$ follow from the standard definition. These operators extend naturally to $\mathscr{S}_{\vbr}^{\vp}\Lambda^{k, l}$ with Cartesian and tensor products (see \cite{Arnold-Boffi-Bonizzoni}), yielding $d^k\colon \mathscr{S}_{\vbr}^{\vp}\Lambda^{k, l}\rightarrow \mathscr{S}_{\vbr}^{\vp}\Lambda^{k+1, l}$ given by
\begin{align*}
d^{k}(u_{1}\otimes \cdots \otimes u_{n})=\sum_{s=1}^{n}(-1)^{|\bm i|_{s-1}}&(u_{1}\otimes \cdots\otimes du_{s}\otimes \cdots\otimes u_{n}),
\end{align*}
for all $u_{l}\in \mathcal{S}^{p_{l}-i_{l}-j_{l}}_{\br_{l}-i_{l}-j_{l}}\Lambda^{i_{l}, j_{l}}(\mathcal{I})$, with $\bm i\in \charset_k$, and $\bm j\in \charset_l$.

\subsection{Tensor product BGG complexes}

To derive the BGG complexes, we first establish a BGG diagram with the spaces obtained above by means of the tensor product construction,
\begin{equation}
 \label{BGGdiagram-general}
\begin{tikzcd}
\cdots \arrow{r} &\mathscr{S}_{\vbr}^{\vp}\Lambda^{i-1, J-1} \arrow{r}{d} &\mathscr{S}_{\vbr}^{\vp}\Lambda^{i, J-1} \arrow{r}{d} & \mathscr{S}_{\vbr}^{\vp}\Lambda^{i+1, J-1} \arrow{r}{} & \cdots \\
\cdots \arrow{r}\arrow[ur]&\mathscr{S}_{\vbr}^{\vp}\Lambda^{i-1, J}
\arrow{r}{d} \arrow[ur, "S^{i-1, J}"]&
\mathscr{S}_{\vbr}^{\vp}\Lambda^{i, J}  \arrow{r}{d} \arrow[ur, "S^{i, J}"]&\mathscr{S}_{\vbr}^{\vp}\Lambda^{i+1, J}  \arrow{r}{}\arrow[ur] &\cdots.
 \end{tikzcd}
\end{equation}

Now, we verify that it satisfies the conditions in Section \ref{sec:review-bgg}, i.e., commutativity $DS=-S\tilde{D}$ and the injectivity/surjectivity conditions.  
Since $\mathscr{S}_{\vbr}^{\vp}\Lambda^{\bs}$ are subspaces of $L^2\Lambda^{\bs}$, the commutativity follows from the continuous level. We will further verify that $S^{k, l}=I\otimes s^{k, l}$ maps between the discrete spaces in the diagram.
We start deriving an explicit expression of the operator $s^{k, l}$ on tensor product alternating forms.

\begin{lemma}\label{lem:s-tensorpro}
Let $\bm i\in \charset_{k}$, $\bm j\in \charset_{l}$ and $\omega_{h}\in \alt^{i_{h}, j_h}\R$ for $h=1, \cdots, n$. Then, $\omega=\omega_{1}\otimes \cdots \otimes \omega_{n}\in\alt^{k}\mathbb{R}^n\otimes \alt^l\mathbb{R}^n$ and there holds
\begin{gather*}
    s^{k, l}\omega
    = \sum_{h=1}^n (-1)^{|\bm{i}|_{h}+|\bm{j}|_{h}}
    \omega_1 \wedge \cdots \wedge 
    s^{i_h,j_h} \omega_h \wedge \cdots \wedge \omega_n.
\end{gather*}
\end{lemma}

\begin{proof}
We prove this lemma for basis forms and conclude the general statement by linearity. Hence, let $\omega_h = (dx^h)^{i_h} \otimes (dx^h)^{j_h}$. Note that, as observed in Remark~\ref{rem:s}, there holds
\begin{gather}
    s^{i_h,j_h}\omega_h=
    \begin{cases}
    dx^h\otimes 1 &\text{if } i_h=0, j_h=1\\
    0&\text{else.}
    \end{cases}
\end{gather}
Using~\eqref{eq:s_iJ}, and recalling~\eqref{eq:Iell}, we find:
\begin{align*}
    &s^{k, l}\omega
    =s^{k, l}( (dx^1)^{i_1} \wedge \cdots \wedge (dx^n)^{i_n} ) \otimes
    ( (dx^1)^{j_1} \wedge \cdots \wedge (dx^n)^{j_n} )\\
    &\quad= \sum_{h=1}^n (-1)^{|\bm{i}|_{h}+1} \delta_{1,j_h}
    \bigl( (dx^h)^{j_h} \wedge (dx^1)^{i_1} \wedge \cdots \wedge (dx^n)^{i_n}\bigr)
    \\ 
    &\quad\qquad\otimes \bigl( (dx^1)^{j_1} \wedge \cdots \wedge \widehat{(dx^h)^{j_h}} \wedge \cdots \wedge (dx^n)^{j_n} \bigr)
    \\
    &\quad = \sum_{h=1}^n (-1)^{|\bm{i}|_{h}+1} \delta_{1,j_h}
    \bigl( (dx^h)^{j_h} \otimes 1\bigr) \wedge \omega_1
    \wedge \cdots
    \wedge \bigl( (dx^h)^{i_h} \otimes 1 \bigr) \wedge \cdots \wedge 
    \omega_n \\
    &\quad = \sum_{h=1}^n (-1)^{|\bm{i}|_{h}+|\bm{j}|_{h}} \delta_{1,j_h}
    \omega_1 \wedge \cdots\wedge
    \bigl( (dx^h)^{j_h} \otimes 1\bigr)
    \wedge \bigl( (dx^h)^{i_h} \otimes 1 \bigr)
    \wedge \cdots \wedge \omega_n.
\end{align*}
Note that 
\begin{multline*}
\delta_{1,j_h}\bigl( (dx^h)^{j_h} \otimes 1 \bigr)
\wedge \bigl( (dx^h)^{i_h} \otimes 1 \bigr)
= \delta_{1,j_h}\bigl((dx^h)^{j_h} \wedge (dx^h)^{i_h}\bigr) \otimes 1
\\= s^{i_h,j_h} \bigl((dx^h)^{i_h} \otimes (dx^h)^{j_h}\bigr)
=  s^{i_h,j_h}\omega_h.
\end{multline*}
In particular, for $j_h=1$, we can assume $i_h=0$, since otherwise we would have a two-form on $\R$, which must be zero. Hence, the lemma is proven.
\end{proof}

\begin{remark}
  Let $u = q \otimes \omega \in \mathscr{S}_{\vbr}^{\vp}\Lambda^{k, l}$. Since $S^{k, l}=I\otimes s^{k, l}$ as in the continuous case, $S^{k,l} u = q \otimes s^{k,l}\omega$. By the definition in~\eqref{def:Sp2}, the polynomial spaces on the left and on the right of $S^{k,l}$ are the same since adding to $i$ subtracts from $j$. This implies that the induced mapping between the coefficient spaces is bijective. Hence, $S^{k, l}\mathscr{S}_{\vbr}^{\vp}\Lambda^{k, l}\subset \mathscr{S}_{\vbr}^{\vp}\Lambda^{k+1, l-1}$ and $S^{k,l}$ inherits surjectivity from $s^{k,l}$.
\end{remark}
 
Following the BGG recipe in Section \ref{sec:review-bgg}, and exploiting the bijectivity of $S^{J-1,J}$, we obtain the discrete spaces
$$
\Upsilon_{h}^{i}:=
\begin{cases}
 \ran(S^{i-1, J})^{\perp}\subset \mathscr{S}_{\vbr}^{\vp}\Lambda^{i, J-1}, \quad i<J;\\
 \ker(S^{i, J})\subset \mathscr{S}_{\vbr}^{\vp}\Lambda^{i, J}, \quad i\geq J, 
\end{cases}
$$
and the operators
$$
\mathscr D^{i}:=
\begin{cases}
 P_{\ran(S^{i-1, J})^{\perp}}d^{i}, \quad i\leq J-1;\\
 d^{i}\circ (S^{i-1, J})^{-1}\circ d^{i}, \quad i=J;\\
d^{i}, \quad i\geq J+1,
\end{cases}
$$
where $h$ is used as a generic index for discrete spaces. The derived discrete BGG complex is 
\begin{equation}\label{reduced-complex}
\begin{tikzcd}
0 \arrow{r}&\Upsilon^{0}_h\arrow{r}{\mathscr{D}^{0}} &\Upsilon^{1}_h \arrow{r}{\mathscr{D}^{1}} &\cdots \arrow{r}{\mathscr{D}^{n-1}} &\Upsilon^{n}_h\arrow{r}& 0.
 \end{tikzcd}
\end{equation}

Thus, we have the algebraic setup in Section~\ref{sec:review-bgg} with the discrete spaces $\mathscr{S}_{\vbr}^{\vp}\Lambda^{k, l}$. 
Following \cite[Theorem 6]{arnold2021complexes}, the main conclusion is the cohomology of the derived BGG complex.

\begin{theorem}\label{thm:ineq-dim}
The dimension of cohomology of \eqref{reduced-complex} is bounded by that of the input complexes, i.e.,
\begin{equation}\label{ineq-dim}
\dim \mathscr{H}^k(\Upsilon^\bs_h, \mathscr{D}^\bs)\leq \dim \mathscr{H}^k(\mathscr{S}_{\vbr}^{\vp}\Lambda^{\bs, J-1}, {d}^\bs)+\dim \mathscr{H}^k(\mathscr{S}_{\vbr}^{\vp}\Lambda^{\bs, J}, {d}^\bs).
\end{equation}
\end{theorem}

\begin{remark}
The construction in \cite{arnold2021complexes} has two levels. First, by the commutativity and the injectivity/surjectivity condition of $S^\bs$, we can conclude with an inequality of dimension as Theorem \ref{thm:ineq-dim}. Second, if more structures are available (referred to as the $K$ operators in \cite{arnold2021complexes} satisfying $S=dK-Kd$), then the inequality becomes an equality. All the examples in this paper satisfy this further condition on the continuous level. Nevertheless, whether \eqref{ineq-dim} is an equality (thus reflecting the correct cohomology) or not is not clear at this stage as the $K$ operators may not map between the right discrete spaces.
\end{remark}

\subsection{Quasi-interpolation operators}\label{sec:quasi_interp}

Interpolation operators are an important theoretical tool for verifying the convergence of numerical schemes in finite element methods and isogeometric analysis. In this section, we obtain interpolation operators for the tensor product BGG complexes \eqref{reduced-complex} from versions in 1D. The key assumption is that we have these operators for both rows in a 1D BGG diagram in \eqref{1D:BGG-diagram}, and the operators satisfy certain conditions when we connect the two rows (\eqref{pis-spi} below). This will be highlighted in Assumption \ref{assump:interp} below. Again, the discussions in this section are abstract in the sense that no splines or finite elements are involved. 
The main conclusion is that we can input bounded commuting maps from continuous spaces to discrete spaces in 1D, and derive the corresponding BGG version in $n$D by using only the algebraic structures of tensor products.  

The discussions will be based on the following assumption.
\begin{assumption}\label{assump:interp}
There exists $\pi^{i, j}: L^{2}\Lambda^{i, j}(\mathcal{I})\to \mathscr{S}^{p}_{\br}\Lambda^{i, j}(\mathcal{I})$, where $i, j=0, 1$ in 1D, that is $L^2$-bounded
$$
\|\pi^{i, j}u\|\leq C\|u\|,
$$
satisfying the commutativity condition
\begin{equation}\label{1D:commutdpi}
  d^{i}\pi^{i, j}=\pi^{i+1, j}d^{i}.  
\end{equation}
 Moreover, we require that 
\begin{equation}\label{pis-spi}
S^{i, j}\pi^{i, j}=\pi^{i+1, j-1}S^{i, j}.
\end{equation}
\end{assumption}
Note that the only nontrivial case for \eqref{pis-spi} is $i=0$ and $j=1$.

In vector proxy,  $S^{0, 1}$ is just identity and $\mathcal{S}^{p-1}_{\br-1}\Lambda^{0, 1}$ is identical to $\mathcal{S}^{p-1}_{\br-1}\Lambda^{1, 0}$. The commutativity of \eqref{pis-spi} would follow from the arguments in Section~\ref{sec:tsp_extderivs} once we use equivalent  quasi-interpolation operators for the two spaces in the 1D BGG diagram connected by the $S^0$ operator. Thus, to satisfy Assumption \ref{assump:interp}, we need to have a consistent set of three bounded commuting quasi-interpolation operators. 

\begin{remark}
\label{rem:pi_ext}
For later convenience, we extend the quasi-interpolation operator $\pi^{i, j}$ by 0, whenever applied to differential forms with index not equal to $(i,j)$. For example, for $i=0$ and $j=1$, $\pi^{0, 1}$ is defined on $L^{2}\Lambda^{0, 1}(\mathcal{I})$ according to Assumption~\ref{assump:interp}, and it is extended to 0, whenever applied to $L^{2}\Lambda^{0, 0}(\mathcal{I})$, $L^{2}\Lambda^{1, 0}(\mathcal{I})$ or $L^{2}\Lambda^{1, 1}(\mathcal{I})$.
\end{remark}

Making use of Remark~\ref{rem:pi_ext} and following~\cite{bonizzoni2021h}, we give the following definition.
\begin{definition}
Given the quasi-interpolation operators in 1D, we define the tensor product quasi-interpolation operator in $n$ dimensions for $k,l=0,\ldots,n$  
$$
\pi^{k,l}_{\otimes n}\colon L^{2}\Lambda^{k, l}_{\otimes n}\to \mathscr{S}^{\vp}_{\vbr}\Lambda^{k, l}
$$ 
as follows:
\begin{gather}
\label{eq:pi_tens}
    \pi^{k,l}_{\otimes n} \coloneqq 
    \sum_{\substack{(i_{i}, \cdots, i_{n})\in \charset_{k}\\ (j_{1}, \cdots, j_{n})\in \charset_{l}}} \pi^{i_1,j_1}\otimes\cdots\otimes \pi^{i_n,j_n}.
\end{gather}
\end{definition}

Consider the particular case where $\pi_{\otimes n}^{k, l}$ is applied to rank one tensor product differential forms. Let $\omega=\omega_{1}\otimes \cdots\otimes \omega_{n}\in L^2\Lambda^{k,l}_{\otimes n}$, with $\omega_{t}\in L^{2}\Lambda^{i_{t}, j_{t}}(\mathcal{I})$, $\bm i=(i_1, \cdots, i_n)\in X_k$, $\bm j=(j_1, \cdots, j_n)\in X_l$. Then, we get
\begin{align*}
\pi_{\otimes n}^{k, l}\omega
&:=(\pi^{i_{1}, j_{1}}\otimes \cdots \otimes \pi^{i_{n}, j_{n}})(\omega_{1}\otimes\cdots\otimes \omega_{n})\\
&=\pi^{i_{1}, j_{1}}\omega_{1}\otimes \cdots \otimes \pi^{i_{n}, j_{n}}\omega_{n}.
\end{align*}

Next, we will show that the tensor product operators defined in~\eqref{eq:pi_tens} are bounded and commute with the differential operators $\mathscr{D}^\bs$ in the BGG complexes \eqref{reduced-complex}. Note that the operators $\mathscr{D}^\bs$ are a composition of $d^\bs$, $P_\ran(S)^\perp$ and $S^{-1}$, depending on the indices \eqref{Dder}. To prove that $\pi_\otimes^\bs$ commutes with $\mathscr{D}^\bs$, we will show that $\pi_\otimes^\bs$ commutes with each one of these operators. 
This will be established in Lemmas~\ref{lem:commut-pi-d1} to~\ref{lem:commut-pi-d3} below.

\begin{lemma}\label{lem:commut-pi-d1}
$\pi_{\otimes n}^\bs$ commutes with $d^\bs$, i.e., 
\begin{equation}\label{dpi}
    d^k\pi_{\otimes n}^{k, l}=\pi_{\otimes n}^{k+1, l}d^k.
\end{equation}
\end{lemma}

\begin{proof}
It is sufficient to show the result on rank one tensor product differential forms. By linearity and density, the conclusion holds for all elements in $L^2\Lambda_{\otimes n}^{k,l}$.
For any $\omega_{t}\in L^{2}\Lambda^{i_{t}, j_{t}}(\mathcal{I})$, $\bm i=(i_1, \cdots, i_n)\in X_k$,  $\bm j=(j_1, \cdots, j_n)\in X_l$, (see \cite[(31)]{bonizzoni2021h})
\begin{equation*}\label{dpicommute}
d^k(\omega_1 \otimes \omega_2 \otimes \cdots \otimes \omega_n)=\sum\limits_{t=1}^n (-1)^{|\bm{i}|_{t-1}} (\omega_1 \otimes \cdots \otimes d^{i_l}\omega_l \otimes \cdots \otimes \omega_n).
\end{equation*}
Then the conclusion follows from the commutativity \eqref{1D:commutdpi} in 1D.
\end{proof}

\begin{lemma}\label{lem:commut-pi-d2}
$\pi_{\otimes n}^\bs$ commutes with $S^\bs$, i.e., 
\begin{equation}\label{pis}
\pi_{\otimes n}^{k+1, l-1} S^{k, l}=S^{k, l}\pi_{\otimes n}^{k, l}, \quad \forall~ 0\leq k\leq n-1,\ 1\leq l\leq n.
\end{equation}
\end{lemma}

\begin{proof}
The conclusion follows from Lemma \ref{lem:s-tensorpro} and the commutativity \eqref{pis-spi} in Assumption \ref{assump:interp}. 
\end{proof}

By Assumption \ref{assump:interp}, in 1D, we have $\pi^{i, j}: L^2\Lambda^{i, j}(\mathcal{I})\to \mathcal{S}_{\br}^p\Lambda^{i, j}(\mathcal{I}), ~i, j=0, 1$. Let $\{dx^{\sigma}\otimes dx^{\mu}\}_{\sigma\in \Sigma(i,n),\mu\in \Sigma(j,n)}$ be a basis of $\alt^{i, j}$. Then $\pi^{i, j}$ induces a unique map $\tilde{\pi}^{i, j}: L^2(\mathcal{I})\to \mathcal{S}_{\br}^p$ between the coefficients, 
satisfying 
\begin{equation}\label{pi-tilde}
\pi^{i, j}=\tilde{\pi}^{i, j}\otimes I.
\end{equation} 

\begin{lemma}
\label{lem:commut-pi-d3}
The following commutativity property holds:
\begin{equation}\label{piP}
\pi^{k, l}_{\otimes n} (I\otimes P_{\ran(s^{k-1, l+1})^{\perp}})= (I\otimes P_{\ran(s^{k-1, l+1})^{\perp}})\pi^{k, l}_{\otimes n}.
\end{equation}
\end{lemma}

\begin{proof}
The claim follows by observing that $\pi^{k, l}_{\otimes n}$ acts on the coefficient function, whereas it is the identity operator on alternating forms; on the contrary, $I\otimes P_{\ran(s^{k-1, l+1})^{\perp}}$ acts only on alternating forms as shown in the following:
\begin{align*}
    &\pi^{k, l}_{\otimes n} \left(I\otimes P_{\ran(s^{k-1, l+1})^{\perp}}\right)
    = \left(\tilde\pi^{k, l}_{\otimes n}\otimes I\right)
    \left(I\otimes P_{\ran(s^{k-1, l+1})^{\perp}}\right)\\ 
    &\quad = \tilde\pi^{k, l}_{\otimes n} \otimes P_{\ran(s^{k-1, l+1})^{\perp}} 
    =\left(I\otimes P_{\ran(s^{k-1, l+1})^{\perp}}\right)
    \left(\tilde\pi^{k, l}_{\otimes n}\otimes I\right)\\
    &\quad = \left(I\otimes P_{\ran(s^{k-1, l+1})^{\perp}}\right)\pi^{k, l}_{\otimes n}.
\end{align*}
\end{proof}

From \eqref{dpi}, \eqref{pis}, \eqref{piP} and the explicit form of the operators in the BGG complexes~\eqref{Dder}, we obtain the main result.
\begin{theorem}
The $\pi_{\otimes n}$ operators commute with $\mathscr{D}$ defined in \eqref{Dder}, i.e., $$
\pi_{\otimes n}^{k+1, l-1}\mathscr{D}^k=\mathscr{D}^k \pi_{\otimes n}^{k, l}.$$
\end{theorem}

Define
\begin{gather*}
    \mathbb{E}^{k}=
    \begin{cases}
    \ran(s^{k-1, l+1})^\perp,\quad k\leq l, \\
    \ker(s^{k, l+1}),\quad k\geq l+1, 
    \end{cases}
\end{gather*}
and the spaces
\begin{gather}
    H(\mathscr{D}^k):=\{u\in L^2\otimes \mathbb{E}^k: \mathscr{D}^k u\in L^2\otimes \mathbb{E}^{k+1}\}
\end{gather}
with the graph norm $\|u\|_{H(\mathscr{D}^k)}^2:=\|u\|^2+\|\mathscr{D}^k u\|^2$.

\begin{theorem}
The operators $\pi_{\otimes n}^{k, l}$ are bounded in $L^2$- and $H(\mathscr{D}^k)$-norms, i.e., there exist positive constants $C$ such that
\begin{xalignat}2
  \|\pi_{\otimes n}^{k, l}u\|&\leq C\|u\|,
  &\forall u&\in L^2\otimes \mathbb{E}^k,
  \\
  \|\pi_{\otimes n}^{k, l}u\|_{H(\mathscr{D}^k)}
  &\leq C\|u\|_{H(\mathscr{D}^k)},
  &\forall u&\in H(\mathscr{D}^k).
\end{xalignat}
\end{theorem}

\begin{proof}
The $L^2$-boundedness of $\pi_{\otimes n}^{k, l}$ is similar to \cite[Lemma 7]{bonizzoni2021h}. In particular, there holds:
\begin{gather*}
    \|\pi_{\otimes n}^{k, l}\|_*
    \leq \sum_{\substack{\bm i\in \charset_{k}\\
    \bm j\in \charset_{l}}}
    \|\pi^{i_1, j_1}\otimes \cdots \otimes \pi^{i_n, j_n}\|_*
    =\sum_{\substack{\bm i\in \charset_{k}\\
    \bm j\in \charset_{l}}}
    \|\pi^{i_1, j_1}\|_*\cdots \|\pi^{i_n, j_n}\|_*,
\end{gather*}
where $\|\cdot\|_*$ denotes the $L^2$-norm of operators.
Canonical argument, see~\cite[Theorem 8.4]{falk2021bubble}, shows that the $L^2$-boundedness and commutativity implies boundedness in the $H(\mathscr{D}^k)$-norm.  
\end{proof}

\section{Spline BGG complexes}
\label{sec:examples_splines}

In this section, we derive spline BGG complexes that satisfy Assumption \ref{assump:1Dcomplex} and Assumption \ref{assump:interp} for arbitrary space dimensions and present the elasticity, Hessian and $\div\div$ complexes as examples. The outline of the section is as follows: In Section 3.1, we state some basic concepts from the spline theory and define a spline BGG complex in 1D along with quasi-interpolation operators for the 1D spline spaces. In Section 3.2, we discretize the diagram (\ref{diagram-3rows2D}) and the 2D version of the diagram \eqref{diagram-4rows} using spline spaces and present the derivation of the 2D stress  complex as a BGG complex. In Section 3.3, we discretize the diagram (\ref{diagram-4rows}) for vector proxies in higher dimensions using spline spaces and present the derivation of the elasticity, Hessian, and $\div\div$ complexes as BGG complexes. 

\subsection{Splines in 1D} \label{sec:splines}

Splines are piecewise polynomial functions that satisfy certain regularity conditions. By convention, their parametric domain is defined as the unit interval $\mathcal{I}=[0,1]$ in 1D. Knot vectors are used to partition $\mathcal{I}$ and define the spline basis functions. A knot vector is a vector $\Sigma$ given by  $\Sigma=[\eta_1,\dots, \eta_{n+p+1}]$ where its components, a.k.a., the knot values (or the knots), satisfy $0\le \eta_{1}\le \eta_2 \le \dots \le  \eta_{n+p+1}\le 1$,  where $p$ denotes the polynomial degree of the spline and $n$ denotes the number of basis functions needed to construct the space $S^p_{\br}(\Sigma)$ defined below. The regularity of a spline defined via $\Sigma$ is given by a vector $\br$ that consists of the regularity values of the spline at the knot values in $\Sigma$. 

Suppose $\Sigma$ includes $N$ distinct knot values and let $\hat{\Sigma}\subseteq \Sigma$ be the set of these distinct knot values. For example, if $\Sigma=\{0.0,0.0,0.2,0.4,1.0,1.0\}$, then $\hat{\Sigma}:=\{0.0,0.2,0.4,1.0\}$. The regularity of a spline at a knot value $\hat{\eta}_i\in \hat{\Sigma}$ is computed by $r_i:=p-m_i$ where $1\le m_i\le p+1$ denotes the number of times the knot value $\hat{\eta}_i$ is repeated in $\Sigma$. 
By utilizing the definition in \cite{schumaker}, we may define the one-dimensional spline space $S^p_{\br}(\Sigma)$ as follows:
\begin{align*}\label{eq:splinesp}
S^p_{\br}&(\Sigma):=\{ \phi:  \exists \phi_i\in \mathcal{P}_p: \phi(x)= \phi_i(x)\ \text{for}\ x\in I_i:=[\hat{\eta}_{i-1},\hat{\eta}_i),\ \hat{\eta}_i\in \hat{\Sigma}, I_i\subset \mathcal{I},\\  &i=1,\dots (N-1),  D^{r_i}\phi_{i-1}(\hat{\eta}_i)= D^{r_i}\phi_{i} (\hat{\eta}_i),\ r_i=0,1,\dots,p-m_i\},
\end{align*}
where $\mathcal{P}_p$ denotes the polynomial space of degree $p$. 

\begin{remark}\label{remarkps}
Note that if $m_i=m$ for some $m\ge 1$ and $\forall \hat{\eta}_i \in \Sigma$, then $S^p_{\br}(\Sigma)$ becomes a regular polynomial space of degree $p$ defined piecewise over the knot intervals in $I$.
\end{remark}

In this paper, we consider splines given by open knot vectors, that is, the case where   $\eta_1=\dots=\eta_{p+1}$ and $\eta_{n+1}=\dots=\eta_{n+p+1}$, and use B-splines (a.k.a. basis splines) since every spline function can be written as a linear combination of B-splines of the same degree \cite{bsplines}. Thus, we refer to B-splines by the term \textit{spline} here.\\
B-spline basis functions of degree $p$ denoted by $\{B_i^p\}$ are defined via the Cox-de Boor formula~\cite{deboor}, which starts with defining the lowest degree basis functions $\{B_i^0(\eta)\}$ and obtains the higher degree B-spline basis functions $\{B_i^p(\eta)\}$ by recursion as follows:
$$
B_i^0(\eta):=
    \left\{\begin{array}{ll}
      1, & \eta_i\le \eta < \eta_{i+1}, \\
      0, & \textrm{otherwise},
    \end{array}\right.
$$
$$
B_i^p(\eta)=\frac{\eta- \eta_i}{\eta_{i+p}-\eta_i} B_i^{p-1}(\eta)+\frac{\eta_{i+p+1}-\eta}{\eta_{i+p+1}-\eta_{i+1}} B_{i+1}^{p-1}(\eta).
$$
Using these basis functions, we may also define $S^p_{\br}(\Sigma)$ as follows:
\begin{align}
S^p_{\br}(\Sigma):=\spn\{ B_i^p(\eta): i=1,\dots,n\}.
\end{align}
Suppose that $r_i\ge 0$ at the internal knots, that is, the B-spline functions are at least continuous at the knots, then the derivative of a B-spline basis function is given by
$$\frac{d}{d\eta}B_{i}^p(\eta)=\frac{p}{\eta_{i+p}-\eta_i} B_{i}^{p-1}(\eta)-\frac{p}{\eta_{i+p+1}-\eta_{i+1}}B_{i+1}^{p-1}(\eta),$$
where $B_{1}^{p-1}(\eta)=B_{n+1}^{p-1}(\eta)=0$, by assumption \cite{veig2014var}. We note that $d:S_{\br}^p(\Sigma)\to S^{p-1}_{\br-1}(\tilde{\Sigma})$ is surjective where $\tilde{\Sigma}=\{\eta_2,\eta_3,\dots, \eta_{n+p}\}$ is an open knot vector obtained by dropping the first and the last knots from the open knot vector $\Sigma$ \cite{veig2014var}.

In the rest of the text, to maintain a compatible notation with the finite element spaces used in the BGG construction in Section \ref{sec:finite-elements}, we use $S^p_{\br}(\mathcal{I})$ to denote a B-spline space of degree $p$ with regularity $\br$ defined over $\mathcal{I}$, excluding the knot vector $\Sigma$ from the notation. Thus, we 
denote the space of $i$-forms with coefficients in $S^p_{\br}(\mathcal{I})$ by $S^p_{\br}\Lambda^i(\mathcal{I})$ as in Section ~\ref{sec:construction}.

By the definition of $S^p_{\br}\Lambda^i(\mathcal{I})$, it follows that the following sequence is a complex and $d$ is onto for any integer-valued vector $\br$ defined as above and any scalar $p\ge 1$. Thus, Assumption~\ref{assump:1Dcomplex} is satisfied.
\begin{equation*}
\begin{tikzcd}
0 \arrow{r}&\mathcal{S}_{\br}^{p}\Lambda^0(\mathcal{I}) \arrow{r}{d}& \mathcal{S}_{\br-1}^{p-1}\Lambda^1(\mathcal{I})\arrow{r}{}&0.
\end{tikzcd}
\end{equation*}

Now, we need to verify Assumption \ref{assump:interp}. We first define the following quasi-interpolation operator as the one defined in \cite{veig2014var}:
 \begin{align}\label{Sinterp0}
     \tilde{\pi}_{0}^p&: L^2(\mathcal{I}) \to S^p_{\br}(\mathcal{I}),&\tilde{\pi}_{0}^p(u):=\sum\limits_{i=1}^{n} \lambda_i^p(u)B_i^p,
 \end{align}
 where each $\lambda_i^p$ is the dual basis functional to 
 the B-spline basis function $B_i^p$(See Theorem 4.41 in \cite{schumaker}). Thus, we have $\lambda_i^p(B_j^p)=\delta_{ij}$ for $i,j=1,2,\cdots,n$. As pointed out in \cite{veig2014var}, this dual basis is used for it enables the satisfaction of the $L^2$-stability of the interpolation operators.
Then, we may uniquely define $\tilde{\pi}_1^{p-1}:L^2(\mathcal{I}) \to S^{p-1}_{\br-1}(\mathcal{I}) $ using (\ref{Sinterp0})
 \begin{align*}
     \tilde{\pi}_{1}^{p-1}v:=\frac{d}{d_x} \tilde{\pi}_0^p \int\limits_{0}^x v(s)\ ds. 
 \end{align*}
Similarly, we define $\tilde{\pi}_2^{p-2}:L^2(\mathcal{I}) \to S^{p-2}_{\br-2}(\mathcal{I}) $ 
  \begin{align*}
     \tilde{\pi}_{2}^{p-2}v:=\frac{d}{d_x} \tilde{\pi}_1^{p-1} \int\limits_{0}^x v(s)\ ds. 
 \end{align*}
 Note that $\tilde{\pi}_{0}^p$ and $\tilde{\pi}_{1}^{p-1}$ are spline preserving, therefore, projections. $\tilde{\pi}_{0}^p$ is $L^2$-stable, and $\tilde{\pi}_{1}^{p-1}$ is $L^2$-stable when the mesh is quasi-uniform \cite{veig2014var}. Similarly, $\tilde{\pi}_{2}^{p-2}$ is spline preserving and $L^2$-stable when the mesh is quasi-uniform. Moreover, these projection operators commute with the differential operators, that is for $i=0,1$, we have
\begin{equation}\label{S:dpicommute}
    \tilde{\pi}_{i+1}^{p-(i+1)} \partial_x v= \partial_x  \tilde{\pi}_{i}^{p-i} v.
\end{equation} 
Then, we use these projection operators to define the quasi-interpolation operators $\pi^{i,j}: L^{2}\Lambda^{i, j}(\mathcal{I})\to \mathcal{S}_{\br-i-j}^{p-i-j}\Lambda^{i, j}(\mathcal{I})$ as in Section~\ref{sec:quasi_interp}.
For example, $\pi^{0,0}: L^2\Lambda^{0, 0}(\mathcal{I})\to \mathcal{S}_{\br}^{p}\Lambda^{0, 0}(\mathcal{I})$ is given by $\pi^{0,0} w=(\tilde{\pi}^{0,0}u) (1\otimes 1)=(\tilde{\pi}_{0}^p u)(1\otimes 1)$ where $w=u(1\otimes 1)\in L^2\Lambda^{0, 0}$. 
Then, let $w=u ( dx^{\sigma_1} \otimes 1)\in L^2\Lambda^{1,0}(\mathcal{I})$ and define $\pi^{1,0}: L^2\Lambda^{1,0}(\mathcal{I})\to S_{\br-1}^{p-1}\Lambda^{1,0}(\mathcal{I})$, as follows:
 \begin{align*}
     \pi^{1,0} w&=\pi^{1,0}u (dx^{\sigma_1}\otimes 1)=(\tilde{\pi}^{1,0} u) (dx^{\sigma_1}\otimes 1)=(\tilde{\pi}_1^{p-1} u) (dx^{\sigma_1}\otimes 1).
 \end{align*}
 We define $\pi^{0,1}=  \pi^{1,0}$ taking advantage of the equivalence of the relevant spline spaces. Similarly, we define $\pi^{1,1}:L^2\Lambda^{1,1} (\mathcal{I})\to S_{\br-2}^{p-2}\Lambda^{1,1}(\mathcal{I})$ as follows: 
 \begin{align*}
     \pi^{1,1} w=\pi^{1,1}u (dx^{\sigma_1}\otimes dx^{\mu_1})&=(\tilde{\pi}^{1,1}u )(dx^{\sigma_1}\otimes dx^{\mu_1})\\&=(\tilde{\pi}_{2}^{p-2} u) (dx^{\sigma_1}\otimes dx^{\mu_1}),
 \end{align*}
 where $w=u(dx^{\sigma_1}\otimes dx^{\mu_1})\in L^2\Lambda^{1,1}(\mathcal{I})$.
 We see that (\ref{1D:commutdpi}) holds due to (\ref{S:dpicommute}).
 Moreover, (\ref{pis-spi}) holds due to the definition of the spaces and interpolation operators, that is, $S^{0,1}\pi^{0,1}=\pi^{1,0}S^{0,1}$.
Thus, the assumptions in Section \ref{sec:review-bgg} are satisfied, and
we can write the 1D BGG diagram as follows:
 \begin{equation}\label{1D:BGG-diagram-2}
\begin{tikzcd}
0 \arrow{r} &\mathcal{S}_{\br}^{p} (\mathcal{I}) \arrow{r}{d} &\mathcal{S}_{\br-1}^{p-1} (\mathcal{I})\arrow{r}{} & 0\\
0 \arrow{r} &\mathcal{S}_{\br-1}^{p-1}(\mathcal{I})\arrow{r}{d} \arrow[ur, "I"]&\mathcal{S}_{\br-2}^{p-2}(\mathcal{I}) \arrow{r}&0.
 \end{tikzcd}
\end{equation}
From (\ref{1D:BGG-diagram-2}), we derive the BGG complex
\begin{equation*}
\begin{tikzcd}
0 \arrow{r}&\mathcal{S}_{\br}^{p}(\mathcal{I}) \arrow{r}{d\circ I \circ d}&\mathcal{S}_{\br-2}^{p-2}(\mathcal{I})\arrow{r}&0.
\end{tikzcd}
\end{equation*} 

\subsection{Splines in higher dimensions}

In $\mathbb{R}^n$, we define the spline spaces over the tensor-product of parametric domain as $\mathcal{I}=[0,1]^n$. The partition of $\mathcal{I}$ is determined by $n$ knot vectors such that an element (with non-zero measure) on the parametric domain is $I_{\bm i}=[\eta_{i_1}^1,\eta_{i_1+1}^1]\otimes \dots \otimes [\eta_{i_n}^n,\eta_{i_n+1}^n]$ where $\eta_{i_j}\neq \eta_{i_j+1}$.
Let $\br_i$ denote the fiber of $\vbr$ in direction $i$, thus being the regularity vector in this coordinate direction.
Then, a spline space in $n$-dimensions is defined via the tensor product of one-dimensional spline spaces, that is, 
$\mathcal S^{\vp}_{\vbr}=\mathcal{S}_{\br_1}^{p_1}\otimes \mathcal{S}_{\br_2}^{p_2}\otimes \dots \otimes \mathcal{S}_{\br_n}^{p_n} $ where $S^{p_i}_{\br_i}$ denotes the spline defined over the $i^{th}$ coordinate direction.

For simplicity, we focus on the case $n=2$. We discretize the 2D version of the second and third rows in the diagram (\ref{diagram-4rows}) via two-dimensional spline spaces as follows:
\begin{equation}\label{diagram-2Dstress}  \adjustbox{scale=0.85,center}{%
\begin{tikzcd}
0 \arrow{r} &\mathcal{S}_{\br_{1}, \br_{2}}^{p_{1}, p_{2}} \arrow{r}{\curl} &\left (
\begin{array}{c}
\mathcal S_{ \br_{1}, \br_{2}-1}^{p_{1}, p_{2}-1}\\
\mathcal S_{\br_{1}-1, \br_{2}}^{p_{1}-1, p_{2}}
\end{array}
\right)\arrow{r}{\div} &\mathcal S_{\br_{1}-1, \br_{2}-1}^{p_{1}-1, p_{2}-1} \arrow{r}{} & 0\\
0 \arrow{r}&
\left (
\begin{array}{c}
\mathcal S_{\br_{1}, \br_{2}-1}^{p_{1}, p_{2}-1}\\
\mathcal S_{\br_{1}-1, \br_{2}}^{p_{1}-1, p_{2}}
\end{array}
\right)
\arrow{r}{\curl} \arrow[ur, "I"]&
\arraycolsep1pt\begin{pmatrix}
\mathcal S_{ \br_{1}, \br_{2}-2}^{p_{1}, p_{2}-2} &\mathcal S_{\br_{1}-1, \br_{2}-1}^{p_{1}-1, p_{2}-1} \\
\mathcal S_{ \br_{1}-1, \br_{2}-1}^{p_{1}-1, p_{2}-1} &\mathcal S_{\br_{1}-2, \br_{2}}^{p_{1}-2, p_{2}}
\end{pmatrix}
 \arrow{r}{\div} \arrow[ur, "2\vskw"]&\left (
\begin{array}{c}
\mathcal S_{\br_{1}-1, \br_{2}-2}^{p_{1}-1, p_{2}-2}\\
\mathcal S_{\br_{1}-2, \br_{2}-1}^{p_{1}-2, p_{2}-1}
\end{array}
\right)  \arrow{r}{} & 0.
 \end{tikzcd}}
\end{equation}

Note that the spaces in the first row of \eqref{diagram-2Dstress} are chosen so as to yield a vector-valued de Rham complex, and we start the second row using the connecting map $S=I$ and complete it in a similar fashion. Since this diagram fulfills the necessary conditions for the BGG construction, we can derive the 2D (stress) elasticity complex as a BGG complex:
 \begin{equation}\label{complex-2Dstress}
\begin{tikzcd}
0 \arrow{r} &\mathcal S_{ \br_{1}, \br_{2}}^{p_{1}, p_{2}} \arrow{r}{\curl\curl} &
\Sigma_{h}
 \arrow{r}{\div}  &\left (
\begin{array}{c}
\mathcal S_{ \br_{1}-1, \br_{2}-2}^{p_{1}-1, p_{2}-2}\\
\mathcal S_{ \br_{1}-2, \br_{2}-1}^{p_{1}-2, p_{2}-1}
\end{array}
\right)  \arrow{r}{} & 0,
 \end{tikzcd}
\end{equation}
where 
$$
\Sigma_{h}:=\left\{\sigma_{h}=
\begin{pmatrix}
\sigma_{11} & \sigma_{12}\\
\sigma_{21} & \sigma_{22}
\end{pmatrix}
\in \begin{pmatrix}
\mathcal S_{ \br_{1}, \br_{2}-2}^{p_{1}, p_{2}-2} &\mathcal S_{ \br_{1}-1, \br_{2}-1}^{p_{1}-1, p_{2}-1} \\
\mathcal S_{ \br_{1}-1, \br_{2}-1}^{p_{1}-1, p_{2}-1} &\mathcal S_{ \br_{1}-2, \br_{2}}^{p_{1}-2, p_{2}}
\end{pmatrix} : \sigma_{12}=\sigma_{21}\right\},
$$
and $\curl\curl v=\Big[\begin{array}{cc}
\frac{\partial^2 v}{\partial x_2^2} & -\frac{\partial^2 v}{\partial x_1 \partial x_2}\\
-\frac{\partial^2 v}{\partial x_1 \partial x_2}&\frac{\partial^2 v}{\partial x_1^2}
\end{array}\Big]$ for $v\in \mathcal S_{ \br_{1}, \br_{2}}^{p_{1}, p_{2}}$.

In the light of Remark~\ref{isomorphic-cx}, we also consider the following discretization of the first two rows of the diagram (\ref{diagram-3rows2D}) via two-dimensional spline spaces.
\begin{equation}\label{diagram-3rows-2D}
      \adjustbox{scale=0.85,center}{%
\begin{tikzcd}[ampersand replacement=\&, column sep=0.3in]
0 \arrow{r} \&\mathcal{S}_{\br_1, \br_2}^{p_1, p_2} \arrow{r}{\grad} \&
\left (
\begin{array}{c}
\mathcal{S}_{\br_1-1,\br_2}^{p_1-1,p_2}\\
\mathcal{S}_{\br_1,\br_2-1}^{p_1,p_2-1}
\end{array}
\right )
 \arrow{r}{\rot} \& \mathcal{S}_{\br_1-1,\br_2-1}^{p_1-1,p_2-1}  \arrow{r}{}\& 0\\
0 \arrow{r}\&
\left (
\begin{array}{c}
\mathcal{S}_{\br_1-1,\br_2}^{p_1-1,p_2}\\
\mathcal{S}_{\br_1,\br_2-1}^{p_1,p_2-1}
\end{array}
\right )
\arrow{r}{\grad} \arrow[ur, "\mathrm{I}"]\&
\left (
\begin{array}{cc}
\mathcal{S}_{\br_1-2,\br_2}^{p_1-2,p_2} &\mathcal{S}_{\br_1-1,\br_2-1}^{p_1-1,p_2-1} \\
\mathcal{S}_{\br_1-1,\br_2-1}^{p_1-1,p_2-1}&\mathcal{S}_{\br_1,\br_2-2}^{p_1,p_2-2} 
\end{array}
\right )  \arrow{r}{\rot} \arrow[ur, "-2\sskw"]\&\left (
\begin{array}{c}
\mathcal{S}_{\br_1-2,\br_2-1}^{p_1-2,p_2-1}\\
\mathcal{S}_{\br_1-1,\br_2-2}^{p_1-1,p_2-2}
\end{array}
\right )   \arrow{r}{} \& 0.
 \end{tikzcd}}
\end{equation}

We note that the diagram (\ref{diagram-3rows-2D}) yields the 2D rotated (stress) elasticity complex:
\begin{equation}\label{complex-2Dstress-rotated}
\begin{tikzcd}[column sep=0.4in]
0 \arrow{r} &\mathcal S_{\br_1, \br_2}^{p_1, p_2} \arrow{r}{\hess} &
\Sigma_h
 \arrow{r}{\rot}  &\left (
\begin{array}{c}
\mathcal S_{\br_1-2, \br_2-1}^{p_1-2, p_2-1}\\
\mathcal S_{\br_1-1, \br_2-2}^{p_1-1, p_2-2}
\end{array}
\right)  \arrow{r}{} & 0,
 \end{tikzcd}
\end{equation}
where $\hat{\Sigma}_h$ is defined by
$$
\hat{\Sigma}_h:=\left\{\sigma_{h}=
\begin{pmatrix}
\sigma_{11} & \sigma_{12}\\
\sigma_{21} & \sigma_{22}
\end{pmatrix}
\in \begin{pmatrix}
\mathcal S_{\br_1-2, \br_2}^{p_1-2, p_2} &\mathcal S_{\br_1-1, \br_2-1}^{p_1-1, p_2-1} \\
\mathcal S_{\br_1-1, \br_2-1}^{p_1-1, p_2-1} &\mathcal S_{\br_1, \br_2-2}^{p_1, p_2-2}
\end{pmatrix} : \sigma_{12}=\sigma_{21}\right\}.
$$

\subsection{Vector proxies for splines in three dimensions}

In this subsection, we discretize the diagram~\eqref{diagram-4rows} via spline spaces in three-dimensions as follows:
 \begin{equation}\label{diagram-4rows-V}
\begin{tikzcd}
0 \arrow{r} &V^{0, 0} \arrow{r}{\grad} &V^{1, 0} \arrow{r}{\curl} & V^{2, 0} \arrow{r}{\div} & V^{3, 0} \arrow{r}{}& 0\\
0 \arrow{r}&
V^{0, 1}
\arrow{r}{\grad} \arrow[ur, "\mathrm{I}"]&
V^{1, 1}  \arrow{r}{\curl} \arrow[ur, "2\vskw"]&V^{2, 1}  \arrow{r}{\div} \arrow[ur, "\tr"]&V^{3, 1}  \arrow{r}{} & 0\\
0 \arrow{r}&
V^{0, 2}
\arrow{r}{\grad} \arrow[ur, "-\mathrm{mskw}"]&
V^{1, 2}  \arrow{r}{\curl} \arrow[ur, "\mathcal{T}"]&V^{2, 2}  \arrow{r}{\div} \arrow[ur, "2\vskw"]&V^{3, 2}  \arrow{r}{} & 0\\
0 \arrow{r} &V^{0, 3}\arrow{r}{\grad} \arrow[ur, "\iota"]&V^{1, 3} \arrow{r}{\curl} \arrow[ur, "-\mskw"]&V^{2, 3} \arrow{r}{\div}\arrow[ur, "I"] & V^{3, 3} \arrow{r}{} & 0,
 \end{tikzcd}
\end{equation}
where $\iota: \mathbb{R}\to \mathbb{M}$ is defined by $\iota u:= uI$, and $\mathcal{T}:\mathbb{M}\to \mathbb{M}$ is defined as $\mathcal{T}u:= u^t-tr(u)I$ \cite{arnold2021complexes}. 
We start with defining $V^{0, 0}:=\mathcal S^{p_{1}, p_{2}, p_{3}}_{  \br_{1},  \br_{2},  \br_{3}}$. The rest of the first row is defined via following the de Rham sequence as follows: 
\begin{gather}
  \label{diagram-V10}
V^{1, 0}=
\begin{pmatrix}
\mathcal S^{p_{1}-1, p_{2}, p_{3}}_{  \br_{1}-1,  \br_{2},  \br_{3}}\\
\mathcal S^{p_{1}, p_{2}-1, p_{3}}_{  \br_{1},  \br_{2}-1,  \br_{3}}\\
\mathcal S^{p_{1}, p_{2}, p_{3}-1}_{  \br_{1},  \br_{2},  \br_{3}-1}
\end{pmatrix},
\quad
V^{2, 0}=
\begin{pmatrix}
\mathcal S^{p_{1}, p_{2}-1, p_{3}-1}_{  \br_{1},  \br_{2}-1,  \br_{3}-1}\\
\mathcal S^{p_{1}-1, p_{2}, p_{3}-1}_{  \br_{1}-1,  \br_{2},  \br_{3}-1}\\
\mathcal S^{p_{1}-1, p_{2}-1, p_{3}}_{  \br_{1}-1,  \br_{2}-1,  \br_{3}}
\end{pmatrix},
\quad
V^{3, 0}=\mathcal S^{p_{1}-1, p_{2}-1, p_{3}-1}_{  \br_{1}-1,  \br_{2}-1,  \br_{3}-1}.
\end{gather}
In the second row, we employ the symmetry of the diagram to obtain $V^{0, 1}= V^{1, 0}$. Then, via differential operations, we derive
\begin{gather}
V^{1, 1}=
  \label{diagram-V11}
\begin{pmatrix}
\mathcal S^{p_{1}-2, p_{2}, p_{3}}_{  \br_{1}-2,  \br_{2},  \br_{3}} &\mathcal S^{p_{1}-1, p_{2}-1, p_{3}}_{  \br_{1}-1,  \br_{2}-1,  \br_{3}}&\mathcal S^{p_{1}-1, p_{2}, p_{3}-1}_{  \br_{1}-1,  \br_{2},  \br_{3}-1} \\
\mathcal S^{p_{1}-1,p_{2}-1,p_{3}}_{  \br_{1}-1, \br_{2}-1, \br_{3}}&\mathcal S^{p_{1}, p_{2}-2, p_{3}}_{  \br_{1},  \br_{2}-2,  \br_{3}}&
\mathcal S^{p_{1}, p_{2}-1, p_{3}-1}_{  \br_{1},  \br_{2}-1,  \br_{3}-1}\\
\mathcal S^{p_{1}-1, p_{2}, p_{3}-1}_{  \br_{1}-1,  \br_{2},  \br_{3}-1} & \mathcal S^{p_{1}, p_{2}-1, p_{3}-1}_{  \br_{1},  \br_{2}-1,  \br_{3}-1}& \mathcal S^{p_{1}, p_{2}, p_{3}-2}_{  \br_{1},  \br_{2},  \br_{3}-2}
\end{pmatrix},
\end{gather}

\begin{gather}
\label{diagram-V21}
V^{2, 1}=
\begin{pmatrix}
\mathcal S^{p_{1}-1, p_{2}-1, p_{3}-1}_{  \br_{1}-1,  \br_{2}-1,  \br_{3}-1} &\mathcal S^{p_{1}-2, p_{2}, p_{3}-1}_{  \br_{1}-2,  \br_{2},  \br_{3}-1}&\mathcal S^{p_{1}-2, p_{2}-1, p_{3}}_{  \br_{1}-2,  \br_{2}-1,  \br_{3}} \\
\mathcal S^{p_{1},p_{2}-2,p_{3}-1}_{  \br_{1}, \br_{2}-2, \br_{3}-1}&\mathcal S^{p_{1}-1, p_{2}-1, p_{3}-1}_{  \br_{1}-1,  \br_{2}-1,  \br_{3}-1}&
\mathcal S^{p_{1}-1, p_{2}-2, p_{3}}_{  \br_{1}-1,  \br_{2}-2,  \br_{3}}\\
\mathcal S^{p_{1}, p_{2}-1, p_{3}-2}_{  \br_{1},  \br_{2}-1,  \br_{3}-2} & \mathcal S^{p_{1}-1, p_{2}, p_{3}-2}_{  \br_{1}-1,  \br_{2},  \br_{3}-2}& \mathcal S^{p_{1}-1, p_{2}-1, p_{3}-1}_{  \br_{1}-1,  \br_{2}-1,  \br_{3}-1}
\end{pmatrix},
\end{gather}
$$
V^{3, 1}=
\begin{pmatrix}
\mathcal S^{p_{1}-2, p_{2}-1, p_{3}-1}_{  \br_{1}-2,  \br_{2}-1,  \br_{3}-1}\\
\mathcal S^{p_{1}-1, p_{2}-2, p_{3}-1}_{  \br_{1}-1,  \br_{2}-2,  \br_{3}-1}\\
\mathcal S^{p_{1}-1, p_{2}-1, p_{3}-2}_{  \br_{1}-1,  \br_{2}-1,  \br_{3}-2}
\end{pmatrix}.
$$
In the third row, again by the symmetry of the diagram we have $V^{0, 2}=V^{2, 0}$, and by transposing $V^{2,1}$ we obtain
$$
V^{1, 2}=
\begin{pmatrix}
\mathcal S^{p_{1}-1, p_{2}-1, p_{3}-1}_{  \br_{1}-1,  \br_{2}-1,  \br_{3}-1} &\mathcal S^{p_{1}, p_{2}-2, p_{3}-1}_{  \br_{1},  \br_{2}-2,  \br_{3}-1} &\mathcal S^{p_{1}, p_{2}-1, p_{3}-2}_{  \br_{1},  \br_{2}-1,  \br_{3}-2} \\
\mathcal S^{p_{1}-2,p_{2},p_{3}-1}_{  \br_{1}-2, \br_{2}, \br_{3}-1}&\mathcal S^{p_{1}-1, p_{2}-1, p_{3}-1}_{  \br_{1}-1,  \br_{2}-1,  \br_{3}-1}&\mathcal S^{p_{1}-1, p_{2}, p_{3}-2}_{  \br_{1}-1,  \br_{2},  \br_{3}-2}\\
\mathcal S^{p_{1}-2, p_{2}-1, p_{3}}_{  \br_{1}-2,  \br_{2}-1,  \br_{3}} & \mathcal S^{p_{1}-1, p_{2}-2, p_{3}}_{  \br_{1}-1,  \br_{2}-2,  \br_{3}}& \mathcal S^{p_{1}-1, p_{2}-1, p_{3}-1}_{  \br_{1}-1,  \br_{2}-1,  \br_{3}-1}
\end{pmatrix}.
$$
The remaining spaces are defined by similar arguments as follows:
$$V^{2, 2}=
\begin{pmatrix}
\mathcal S^{p_{1}, p_{2}-2, p_{3}-2}_{  \br_{1},  \br_{2}-2,  \br_{3}-2} &\mathcal S^{p_{1}-1, p_{2}-1, p_{3}-2}_{  \br_{1}-1,  \br_{2}-1,  \br_{3}-2}&\mathcal S^{p_{1}-1, p_{2}-2, p_{3}-1}_{  \br_{1}-1,  \br_{2}-2,  \br_{3}-1} \\
\mathcal S^{p_{1}-1,p_{2}-1,p_{3}-2}_{  \br_{1}-1, \br_{2}-1, \br_{3}-2}&\mathcal S^{p_{1}-2, p_{2}, p_{3}-2}_{  \br_{1}-2,  \br_{2},  \br_{3}-2}&
\mathcal S^{p_{1}-2, p_{2}-1, p_{3}-1}_{  \br_{1}-2,  \br_{2}-1,  \br_{3}-1}\\
\mathcal S^{p_{1}-1, p_{2}-2, p_{3}-1}_{  \br_{1}-1,  \br_{2}-2,  \br_{3}-1} & \mathcal S^{p_{1}-2, p_{2}-1, p_{3}-1}_{  \br_{1}-2,  \br_{2}-1,  \br_{3}-1}& \mathcal S^{p_{1}-2, p_{2}-2, p_{3}}_{  \br_{1}-2,  \br_{2}-2,  \br_{3}}
\end{pmatrix},
$$
$$
V^{3, 2}=
\begin{pmatrix}
\mathcal S^{p_{1}-1, p_{2}-2, p_{3}-2}_{  \br_{1}-1,  \br_{2}-2,  \br_{3}-2}\\
\mathcal S^{p_{1}-2, p_{2}-1, p_{3}-2}_{  \br_{1}-2,  \br_{2}-1,  \br_{3}-2}\\
\mathcal S^{p_{1}-2, p_{2}-2, p_{3}-1}_{  \br_{1}-2,  \br_{2}-2, \br_{3}-1}
\end{pmatrix}.
$$
Finally, in the last row, we obtain the first three spaces by symmetry and derive the last one via the $\div$ operator:
$$ 
V^{3, 3}=\mathcal S^{p_{1}-2, p_{2}-2, p_{3}-2}_{  \br_{1}-2,  \br_{2}-2,  \br_{3}-2}.$$

Now, we can define the elasticity complex~\eqref{sequence:hs} in three dimensions via the second and third rows of the diagram~\eqref{diagram-4rows-V}
\begin{equation}
  \label{diagram-3Delasticity}
  \begin{tikzcd}
  0 \arrow{r} &V^{0, 1} \arrow{r}{\grad} &V^{1, 1} \arrow{r}{\curl} & V^{2, 1} \arrow{r}{\div} & V^{3, 1} \arrow{r}{}& 0\\
  0 \arrow{r}&
  V^{0, 2}
  \arrow{r}{\grad} \arrow[ur, "-\mathrm{mskw}"]&
  V^{1, 2}  \arrow{r}{\curl} \arrow[ur, "\mathcal{T}"]&V^{2, 2}  \arrow{r}{\div} \arrow[ur, "2\vskw"]&V^{3, 2}  \arrow{r}{} & 0.
 \end{tikzcd}
\end{equation}
It is easy to check the relations such as  $\mskw V^{0, 2}\subset V^{1, 1}$, $\mathcal{T} V^{1, 2}= V^{2, 1}$ (by definition), and $\vskw (V^{2, 2})\subset V^{3, 1}$. Therefore the diagram \eqref{diagram-3Delasticity} satisfies all the assumptions for the BGG construction \cite{arnold2021complexes}. We can read out the 3D elasticity complex from the second and third rows of \eqref{diagram-4rows-V} as follows:
 \begin{equation} 
\begin{tikzcd}
0 \arrow{r} &V^{0, 1} \arrow{r}{\deff} &V^{1, 1}\cap \mathbb{S} \arrow{r}{\inc} & V^{2, 2}\cap \mathbb{S} \arrow{r}{\div} & V^{3, 2} \arrow{r}{}& 0,
 \end{tikzcd}
\end{equation}
Similarly, from the first and second rows of \eqref{diagram-4rows-V}, we obtain the Hessian complex:
 \begin{equation} 
\begin{tikzcd}
0 \arrow{r} &V^{0, 0} \arrow{r}{\hess } &V^{1, 1}\cap \mathbb{S} \arrow{r}{\curl} & V^{2, 1}\cap \mathbb{T} \arrow{r}{\div} & V^{3, 1} \arrow{r}{}& 0,
 \end{tikzcd}
\end{equation}
and from the third and fourth rows of \eqref{diagram-4rows-V}, we obtain the $\div\div$ complex:
 \begin{equation}  
\begin{tikzcd}
0 \arrow{r} &V^{0, 2} \arrow{r}{\dev\grad } &V^{1, 2}\cap \mathbb{T} \arrow{r}{\sym\curl} & V^{2, 2}\cap \mathbb{S} \arrow{r}{\div\div} & V^{3, 3} \arrow{r}{}& 0.
 \end{tikzcd}
\end{equation}

\section{Finite element BGG complexes}
\label{sec:examples_fe}

Within this Section, we first present finite element (FE) de Rham complexes in 1D fulfilling Assumption~\ref{assump:1Dcomplex}. Later on, FE BGG complexes are introduced and the obtained results are compared to existing FEs in the literature.

\subsection{Finite elements in 1D}\label{sec:finite-elements}

When we turn from splines to FEs, the view changes from global to local. Again, we consider a subdivision of the interval of interest into subintervals.
We define polynomial spaces on each subinterval by affine mapping of a polynomial space on the reference interval $[0,1]$ to the actual interval.
What constituted the knot vector for splines are now the interfaces between the subintervals together with so-called node functionals, which are also defined by mapping from the reference interval. It is common standard to assume that continuity conditions between all subintervals are equal. Thus, a FE space is defined by
\begin{enumerate}
    \item The subdivision into subintervals
    \item The polynomial space and node functionals on the reference interval
\end{enumerate}
Note that in one space dimension this corresponds to a spline space where each knot has the same multiplicity. In higher dimensions, the methods differ by the fact that the FE version does not require tensor product meshes, just each cell must be a tensor product.

Following these remarks, $\mathcal{S}^q_{r}(\interval)$ denotes the space of polynomials on the reference interval $\interval$ of degree $q$ equipped with node functionals which establish regularity of degree $r$ at the interfaces between subintervals. In the examples below, $r=-1$ will refer to discontinuous FEs, $r=0$ to continuous, and $r=1$ to continuously differentiable elements.

We now derive the FE discretization for the one-di\-men\-sion\-al BGG diagram~\eqref{diagram-1D} as well as the one-dimensional BGG complex~\eqref{BGG-complex1D}. First, we observe that the regularity index of the last element in the complex is two less compared to the first element. As we cannot go below $L^2$, this implies that we must start with (at least) $H^2$ regularity on the left. 
The proxy field FE-BGG diagram reads:
\begin{equation}
    \label{1D:BGG-diagram-FEM}
    \begin{tikzcd}
        0 \arrow{r} &W^{0,0} \arrow{r}{\partial_x} &W^{1,0} \arrow{r}{} & 0\\
        0 \arrow{r} &W^{0,1}\arrow{r}{\partial_x} \arrow[ur, "I"]&W^{1,1} \arrow{r}{} & 0.
    \end{tikzcd}
\end{equation}
In detail, the spaces in the FE complex are (see~\cite{bonizzoni2021h,Bonizzoni-Kanschat2022}):
\begin{itemize}
    \item $W^{0,0}$ is the space of polynomials  in $\mathcal{S}^{q}_{\br}(\interval)$ defined by a modified set of Hermite interpolation conditions ensuring continuously differentiable transitions to the neighboring intervals, complemented with a set of interior moments. This results in the following set of $q+1$ node functionals:
    \begin{align}
        \label{eq:nodal00-1}
        \nodal^{0,0}_{2i+1} (u) = \partial_x^{i+1} u(0),\quad
            \nodal^{0,0}_{2i+2} (u) = \partial_x^{i+1} u(1), 
            &\quad i=0,\ldots,r-1,\\
        \label{eq:nodal00-2}    
        \nodal^{0,0}_{2r+i} (u) = \int_0^1 \ell_{i-1} \partial_x u\,\dx,
            &\quad i=1,\dots,q-2r,\\
        \label{eq:nodal00-3}
        \nodal^{0,0}_{q+1}(u) = u(1)+u(0),&    
    \end{align}
where $\ell_i$ denotes the Legendre polynomial of degree $i$ on the interval $[0,1]$, normalized with the condition $\ell_i(1)=1$.

    \item 
    $W^{1,0}$ is the space of polynomials in $\mathcal{S}^{q-1}_{r-1}(\interval)$ equipped with the following set of $q$ node functionals
    \begin{align}
        \label{eq:nodal10-1}
        \nodal^{1,0}_{2i+1} (v) = \partial_x^i v(0), \quad
            \nodal^{1,0}_{2i+2} (v) = \partial_x^i v(1),
            &\quad i=0,\ldots,r-1,\\
        \label{eq:nodal10-2}
        \nodal^{1,0}_{2r+i} (v) = \int_0^1 \ell_{i-1} v\, \dx,
        &\quad i=1,\dots,q-2r.
    \end{align}

    \item 
    $W^{0,1}$ is the space of polynomials in $\mathcal{S}^{q-1}_{r-1}(\interval)$ equipped with the set of node functionals $\{\nodal^{0,1}_j\}_{j=1}^q$ defined as in~\eqref{eq:nodal00-1}-\eqref{eq:nodal00-2}-\eqref{eq:nodal00-3} with $r$ and $q$ replaced by $r-1$ and $q-1$, respectively.

    \item 
    $W^{1,1}$ is the space of polynomials in $\mathcal{S}^{q-2}_{r-2}(\interval)$ equipped with the set of node functionals $\{\nodal^{1,1}_j\}_{j=1}^{q-1}$ defined as in~\eqref{eq:nodal10-1}-\eqref{eq:nodal10-2} with $r$ and $q$ replaced by $r-1$ and $q-1$, respectively.
\end{itemize}

The unisolvence of these finite elements was proven in~\cite[Section 3.3]{Bonizzoni-Kanschat2022}.
It is easily verified that each row in~\eqref{1D:BGG-diagram-FEM} forms a complex and $\partial_x$ is onto, so that Assumption~\ref{assump:1Dcomplex} is satisfied.

\begin{remark}
\label{rem:change-of-basis}
    The sets of node functionals in equations~\eqref{eq:nodal00-2} and~\eqref{eq:nodal00-3} do not contain $u(0)$ and $u(1)$ as such.
    \textcolor{purple}{Since} these are used in standard finite element definitions to ensure continuity at the interface of two elements\textcolor{purple}{, we have to ensure that our node functionals still yield a conforming method. Note, that they} encompass the linear combinations
    \begin{gather}
    \label{eq:commuting-pair}
        \nodal^{0,0}_{q+1}(u) = u(1)+u(0)
        \quad\text{and}\quad
        \nodal^{0,0}_{2r+1}(u) = \int_0^1 u'(x) \dx = u(1)-u(0).
    \end{gather}
    As this is a change of basis in the dual space, it is easily realized that the interpolants obtained by the commuting interpolation operators and by the standard ones, respectively, are equal. Hence, the commuting interpolation operators are used for the analysis only, while the standard ones \textcolor{purple}{yield conformity and} are used for practical implementation as in Figures~\ref{fig:V22} to~\ref{fig:V11} below.
    \textcolor{purple}{ More details on this remark can be found in Appendix~\ref{app:interpolation}}
\end{remark}

Let $n_{k,l}$ be the dimension of the shape function space $W^{k,l}$.
Given the above-mentioned node functionals, we can introduce the corresponding dual basis functions $\{\psi^{k,l}_i\}_{i=1,\ldots,n_{k,l}}$ for $W^{k,l}$, fulfilling
\begin{gather}
  \label{eq:deltaij}
  \nodal^{k,l}_j(\psi^{k,l}_i)=\delta_{ij},
  \qquad k,l=0,1,
\end{gather}
for all admissible values of $i$ and $j$. Then, the canonical interpolation operators on the polynomial spaces are defined on the reference interval as
\begin{gather}
    \label{1D:interp-FEM1}
    \interp^{k,l}(u):=\sum_{i=1}^{n_{k,l}} \nodal^{k,l}_i(u)\psi^{k,l}_i.
\end{gather}
The basis functions $\{\psi^{k,l}_i\}$ together with the node functionals $\{\nodal^{k,l}_i\}$ for $k,l=0,1$ and all admissible values of $i$, fulfill the assumptions of the commuting Lemma 2 in~\cite{Bonizzoni-Kanschat2022}. Hence, the canonical interpolation operators~\eqref{1D:interp-FEM1} commute with the  derivative, i.e., for all $u \in C^\infty(\interval)$, there holds
\begin{gather*}
  \partial_x \interp^{0,l} u = \interp^{1,l} \partial_x u, \qquad l=0,1.
\end{gather*}
\begin{lemma}
\label{lem:proxy-interpolation}
The canonical interpolation operator commutes with the identity operator in~\eqref{1D:BGG-diagram-FEM}, that is, for any $u\in C^{\infty}(\interval)$ there holds
\begin{gather}
    \label{eq:fe-commute-canonical}
    \interp^{0,1} u = \interp^{1,0} u.
\end{gather}
\end{lemma}

\begin{proof}
First, we observe that the node functionals in~\eqref{eq:nodal00-2} for $i>1$ can be transformed by integration by parts into linear combinations of those in~\eqref{eq:nodal10-2}. In the case $i=1$ they simply reduce to the difference of the values in the end points.
Thus, the sets of node functionals for $W^{1,0}$ and for $W^{0,1}$ can each be obtained as linear combinations of the other set. We decompose
\begin{gather}
    C^\infty = \mathbb P_{q-1} \oplus \mathbb P_{q-1}^\perp,
\end{gather}
where $\mathbb P_{q-1}^\perp$ is the polar set of the node functionals, namely
\begin{align}
\mathbb P_{q-1}^\perp
  &= \bigl\{ u\in C^\infty \;\big|\;
  \nodal^{1,0}_i(u) = 0, \;i=1,\dots,q \}\\
  &= \bigl\{ u\in C^\infty \;\big|\;
  \nodal^{0,1}_i(u) = 0, \;i=1,\dots,q \}.
\end{align}
Hence, decomposing $u\in C^\infty$ according to this direct sum as $u=p+u^\perp$ and 
exploiting that both interpolation operators are projections onto $\mathbb P_{q-1}$, there holds
\begin{gather}
    \interp^{0,1} u = \interp^{0,1} p = p = \interp^{1,0} p = \interp^{1,0} u. 
\end{gather}
\end{proof}

The node functionals $\{\nodal^{k,l}_i\}$ are well-defined for smooth functions. However, as outlined in Remark 3 in~\cite{Bonizzoni-Kanschat2022}, weighted node functionals $\{\wnodal^{k,l}_i\}$ can be employed to obtain an $L^2$-bounded quasi-interpolation operator
\begin{gather}
  \label{1D:quasi-interp-FEM}
  \widetilde{\winterp}^{k,l}(u):=\sum_{i=1}^{n_{k,l}}
  \wnodal^{k,l}_i(u) \psi^{k,l}_i,
  \qquad k,l=0,1,
  \qquad u\in L^2(I),
\end{gather}
which commutes with the exterior derivative. We recall that $\widetilde{\winterp}^{k,l}$ is obtained by averaging over canonical interpolation operators~\eqref{1D:interp-FEM1} on perturbed intervals. Since~\eqref{eq:fe-commute-canonical} holds for each of them, and the averaging weights are chosen consistently, we conclude that the quasi-interpolation operators $\widetilde{\winterp}^{k,l}$ commute with the operators $S^{k,l}$. Hence, they satisfy Assumption~\ref{assump:interp}.
Note that applying Schöberl's trick, see for instance section 5.3 in~\cite{bonizzoni2021h}, we can even obtain a projection onto the discrete space.

\subsection{Finite elements in three dimensions}\label{sec:finite-elements-3D}
Tensor products of one-di\-men\-sion\-al functions are simply defined as $[u\otimes v\otimes w](x,y,z) = u(x)v(y)w(z)$. Thus, the definition of tensor product polynomials is straightforward.
Functions of this type are called rank-1 tensors. Since rank-1 tensors span the tensor product space, it is sufficient to define the tensor product of node functionals on such tensors and then extend them by linearity to the whole space.
Thus, for any indices $i,j,k$, a node functional $\nodal_{ijk}$ is defined by
\begin{gather}
  \label{eq:nodal-tensor}
    \nodal_{ijk}(u\otimes v\otimes w)
    = [\nodal_i\otimes\nodal_j\otimes\nodal_k](u\otimes v\otimes w)
    = \nodal_i(u)\nodal_j(v)\nodal_k(w).
\end{gather}
Hence, each node functional applies to the spatial variable $x$, $y$, or $z$ according to its position in the tensor product.

The unisolvence of the tensor product elements follows immediately from the same property of the one-dimensional fibers: let $\{\phi_\alpha\}$ be the basis of a one-dimensional element such that the interpolation condition~\eqref{eq:deltaij} holds. Then, by~\eqref{eq:nodal-tensor},
\begin{gather*}
    \nodal_{ijk}(\phi_\alpha\otimes\phi_\beta\otimes\phi_\gamma)
    = \delta_{i\alpha}\delta_{j\beta}\delta_{k\gamma}.
\end{gather*}
Using the fact that the set of rank-1 tensors of basis functions forms a basis of the tensor product space, we obtain a tensor product basis which is dual to the node functionals. We finish this discussion noting that the argument holds for anisotropic tensor product as well.

Finally, we observe that the tensor product of de Rham subcomplexes is a subcomplex of the de Rham complex on the Cartesian product domain (see~\cite{Arnold-Boffi-Bonizzoni}). Hence follows the conformity of the FE BGG complexes.

\begin{example}[Node functionals in 2D]
Let for instance in two dimensions $\nodal_1$ and $\nodal_2$ be node functionals in one dimension. Then
$[\nodal_1\otimes \nodal_2](u\otimes v) =\nodal_1(u)\nodal_2(v)$.
Choosing node functionals $\nodal_1(u) = u'(0)$, $\nodal_2(u) = u(0)$, and $\nodal_3(u)=\int u\dx$ we obtain by this construction in two dimensions
\begin{gather}
\label{eq:tensor-product-dofs}
\begin{aligned}
    [\nodal_1\otimes \nodal_1] (f) &= \partial_{xy} f(0,0),\qquad
    & [\nodal_1\otimes \nodal_3] (f) &= \int_0^1 \partial_x f(0,y)\,dy,\\
    [\nodal_1\otimes \nodal_2] (f) &= \partial_{x} f(0,0),
    & [\nodal_2\otimes \nodal_3] (f) &= \int_0^1 f(0,y)\,dy,\\
    [\nodal_2\otimes \nodal_1] (f) &= \partial_{y} f(0,0),
    & [\nodal_3\otimes \nodal_3] (f) &= \int_0^1 f(x,y)\,dx\,dy.
\end{aligned}
\end{gather}
\end{example}

Now we have set the stage for studying specific FEs for BGG complexes. Let us begin with the lowest regularity ($r=1$) element family for the $\div\div$ complex obtained by our method. This complex combines the last two rows of the diagram~\eqref{diagram-4rows-V}. Starting with $V^{3,3}$, we chose $\br_1 = \br_2 = \br_3 = 1$ to obtain an $L^2$-conforming space. Focusing on the isotropic case $p_1=p_2=p_3=p$, we obtain the following polynomial spaces for the $\div\div$ complex:
\begin{xalignat}2
  V^{0,2} &=
  \begin{pmatrix}
    \mathcal S^{p,p-1,p-1}_{1,0,0}\\
    \mathcal S^{p-1,p,p-1}_{0,1,0}\\
    \mathcal S^{p-1,p-1,p}_{0,0,1}
  \end{pmatrix}
  &V^{1,2} &=
  \begin{pmatrix}
    \mathcal S^{p-1, p-1, p-1}_{0,0,0}
    &\mathcal S^{p,p-2,p-1}_{1,-1,0}
    &\mathcal S^{p, p-1, p-2}_{1,0,-1} \\
    \mathcal S^{p-2,p,p-1}_{-1,1,0}
    &\mathcal S^{p-1,p-1,p-1}_{0,0,0}
    &\mathcal S^{p-1,p,p-2}_{0,1,-1}\\
    \mathcal S^{p-2,p-1,p}_{-1,0,1}
    &\mathcal S^{p-1,p-2,p}_{0,-1,1}
    &\mathcal S^{p-1,p-1,p-1}_{0,0,0}
  \end{pmatrix}\\
  V^{3,3} &= \mathcal S^{p-2,p-2,p-2}_{-1,-1,-1}
  &  V^{2,2} &=
  \begin{pmatrix}
    \mathcal S^{p,p-2, p-2}_{1,-1,-1}
    &\mathcal S^{p-1,p-1,p-2}_{0,0,-1}
    &\mathcal S^{p-1,p-2,p-1}_{0,-1,0} \\
    \mathcal S^{p-1,p-1,p-2}_{0,0,-1}
    &\mathcal S^{p-2,p,p-2}_{-1,1,-1}&
    \mathcal S^{p-2,p-1,p-1}_{-1,0,0}\\
    \mathcal S^{p-1,p-2,p-1}_{0,-1,0}
    &\mathcal S^{p-2,p-1,p-1}_{-1,0,0}
    &\mathcal S^{p-2,p-2,p}_{-1,-1,1}
\end{pmatrix}
\end{xalignat}
Thus, we have obtained exactly the same polynomial spaces as~\cite{hu2022new} as a special case of our theory. In particular, the three spaces on the diagonal of $V^{1,2}$ are identical, thus facilitating trace free matrices. Similarly, $V^{2,2}$ is suitable for symmetric matrices.

The degrees of freedom of these spaces can be read from the lower indices and can be constructed by tensor products as in equation~\eqref{eq:tensor-product-dofs}. Beginning from the right, $V^{3,3}$ is simply the space of discontinuous tensor product polynomials $\mathbb{Q}_{p-2}$ with only volume node functionals, implemented as moments with respect to the same space,
\begin{align*}
    &\int_K p\phi\dx & \phi&\in \mathbb{Q}_{p-2}(K).
\end{align*}
For $V^{2,2}$, the $\div\div$-conforming space of symmetric matrices, we obtain for the diagonal element $\sigma_{ii}$ the conditions, that the piecewise polynomials are continuously differentiable in $x_i$-direction and discontinuous in the other two coordinate directions. Thus, we introduce the set of node functionals
\begin{align}
    &\int_F (\n^T\sigma\n) \phi \dd s & \phi&\in \mathbb{Q}_{p-2}(F),\\
    &\int_F (\n^T\partial_n\sigma\n) \phi \dd s & \phi&\in \mathbb{Q}_{p-2}(F),
\end{align}
where $F$ runs through all faces of the reference cube $K$. Here, we used that on each Cartesian face, $\n^T\sigma\n$ selects the diagonal element $\sigma_{ii}$ where the unit vector $e_i$ is orthogonal to $F$. These degrees of freedom are shown for the lowest order case in the top row of figure~\ref{fig:V22}.
\begin{figure}
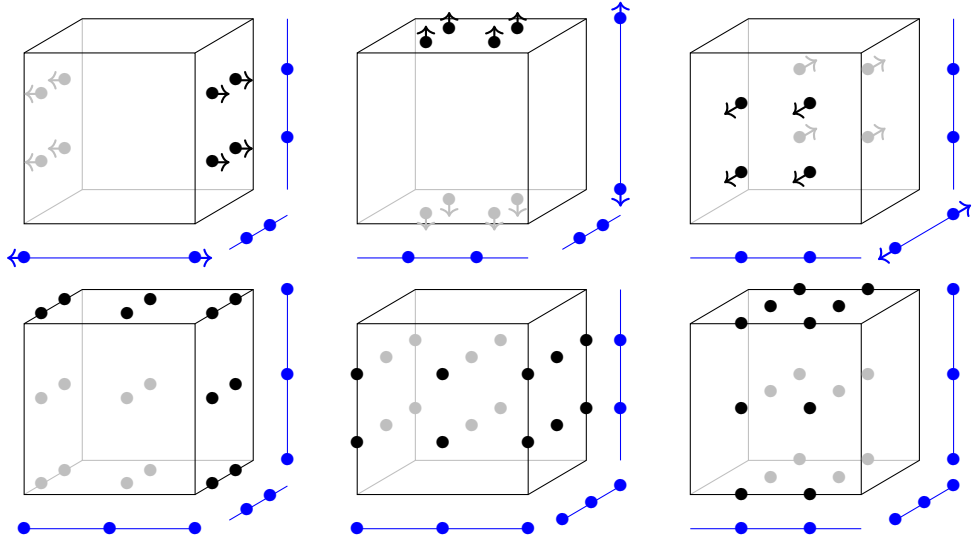

    \centering
    \includegraphics[width=.3\textwidth]{fig/311.tikz}
    \hfill
    \includegraphics[width=.3\textwidth]{fig/131.tikz}
    \hfill
    \includegraphics[width=.3\textwidth]{fig/113.tikz}
    
    \includegraphics[width=.3\textwidth]{fig/221.tikz}
    \hfill
    \includegraphics[width=.3\textwidth]{fig/212.tikz}
    \hfill
    \includegraphics[width=.3\textwidth]{fig/122.tikz}
    \caption{The degrees of freedom of the lowest order version of the $\div\div$-conforming space $V^{2,2}$ and the fibers of the tensor products in blue. Top row the diagonal entries $\sigma_{11}, \sigma_{22},\sigma_{33}$. Bottom row $\sigma_{12}=\sigma_{21}$, $\sigma_{13}=\sigma_{31}$, and $\sigma_{23}=\sigma_{32}$. Dots indicate function values, arrows indicate directional derivatives. Moments are visualized by quadrature points.}
    \label{fig:V22}
\end{figure}
For visualization purposes, the moments on faces are displayed as equivalent values in quadrature points.
At the bottom of figure~\ref{fig:V22}, we show how the degrees of freedom for the off-diagonal elements are constructed as tensor products of their one-dimensional fibers in the lowest order case.

The diagonal elements of a tensor $\sigma$ in the space $V^{1,2}$ are each from the standard continuous, isotropic FE space $\mathbb{Q}_{p-1}$.
The fact, that the corresponding matrix spaces in equations~\eqref{grad-grad0} and~\eqref{div-div0} are trace free implies that one component of the diagonal is determined by the other two. Hence, one shape function set can be eliminated.
The off-diagonal entries are anisotropic in polynomial degree and degrees of freedom. Instead of writing down complicated formulas for node values and test spaces, we show their construction in the lowest order case in figure~\ref{fig:V12}. 
\begin{figure}
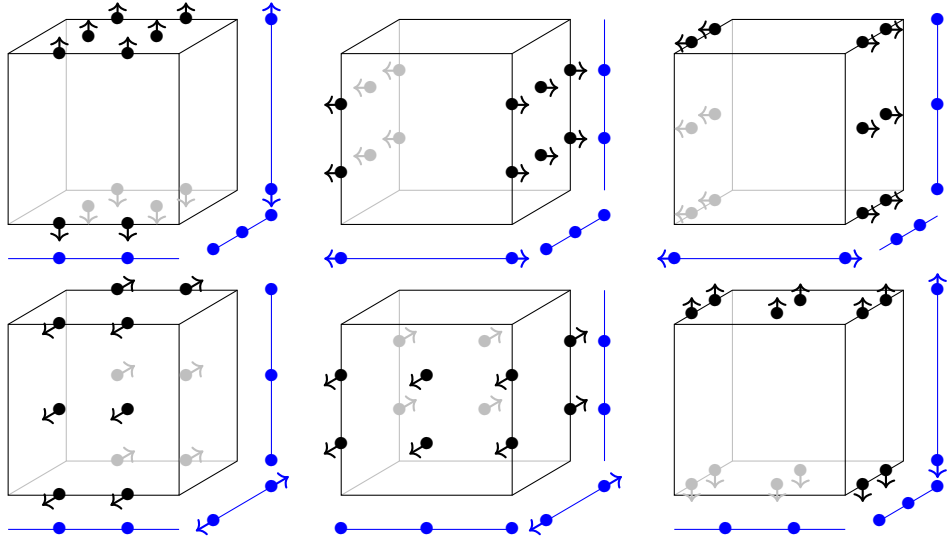

    \centering
    \includegraphics[width=.3\textwidth]{fig/132.tikz}
    \hfill
    \includegraphics[width=.3\textwidth]{fig/312.tikz}
    \hfill
    \includegraphics[width=.3\textwidth]{fig/321.tikz}
    
    \includegraphics[width=.3\textwidth]{fig/123.tikz}
    \hfill
    \includegraphics[width=.3\textwidth]{fig/213.tikz}
    \hfill
    \includegraphics[width=.3\textwidth]{fig/231.tikz}
    \caption{The degrees of freedom of the lowest order version of the symcurl-conforming space $V^{1,2}$ and the fibers of the tensor products in blue. First row for entries $\sigma_{21},\sigma_{12},\sigma_{13}$ and second row for $\sigma_{31},\sigma_{32},\sigma_{32}$. Dots indicate function values, arrows indicate directional derivatives. Moments are visualized by quadrature points.}
    \label{fig:V12}
\end{figure}
Each of the entries is a tensor product of a polynomial of degree 3, one of degree 2 and one of degree 1, and the six off-diagonal entries traverse all possible combinations of these. The node functionals are those of Hermite and Lagrange interpolation, respectively.

The tensor product construction also yields finite elements for the elasticity complex. The degrees of freedom of the lowest order strain element of our construction is displayed in figure~\ref{fig:V11}.
\begin{figure}
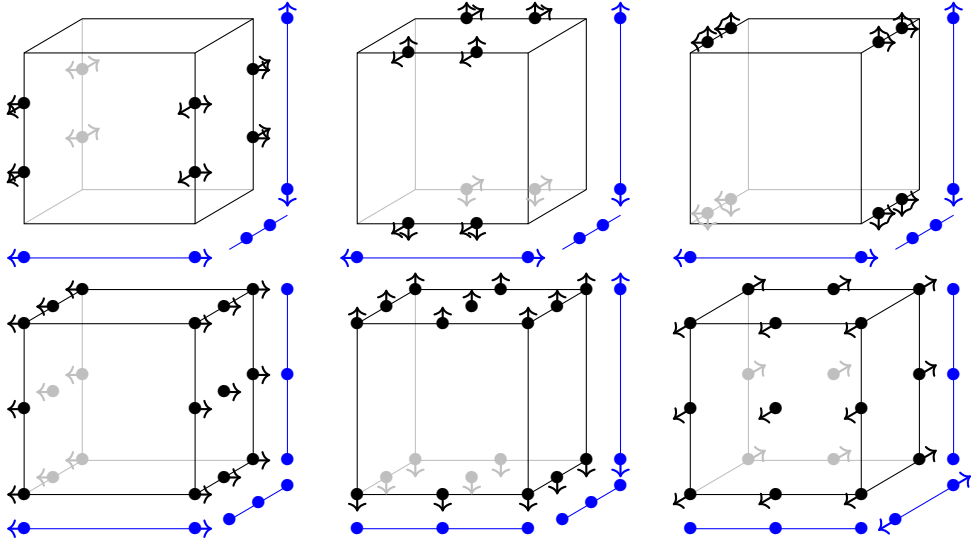

    \centering
    \includegraphics[width=.3\textwidth]{fig/133.tikz}
    \hfill
    \includegraphics[width=.3\textwidth]{fig/313.tikz}
    \hfill
    \includegraphics[width=.3\textwidth]{fig/331.tikz}
    
    \includegraphics[width=.3\textwidth]{fig/322.tikz}
    \hfill
    \includegraphics[width=.3\textwidth]{fig/232.tikz}
    \hfill
    \includegraphics[width=.3\textwidth]{fig/223.tikz}
    \caption{Degrees of freedom for the lowest order strain element $V^{1,1}$ of the elasticity complex. Diagonal entries $\sigma_{11},\sigma_{22},\sigma_{33}$ in the top row. Off-diagonal entries $\sigma_{23}=\sigma_{32}$, $\sigma_{13}=\sigma_{31}$, and $\sigma_{12}=\sigma_{21}$ in the bottom row. Pairs of arrows indicate first order and mixed second order derivatives.}
    \label{fig:V11}
\end{figure}
The corresponding stress element is $V^{2,2}$ in figure~\ref{fig:V22}. Comparing to~\cite{hu2014simple}, we see that their element in our notation is
\begin{gather}
  \label{eq:HMZ1}
    V^{2, 2}_{\text{HMZ}}=
\begin{pmatrix}
\mathcal S^{2,0,0}_{0,-1,-1} &\mathcal S^{1,1,0}_{0,0,-1}&\mathcal S^{1,0,1}_{0,-1,0} \\
\mathcal S^{1,1,0}_{0,0,-1}&\mathcal S^{0,2,0}_{-1,0,-1}&\mathcal S^{0,1,1}_{-1,0,0}\\
\mathcal S^{1,0,1}_{0,-1,0} & \mathcal S^{0,1,1}_{-1,0,0}&\mathcal S^{0,0,2}_{-1,-1,0}
\end{pmatrix},
\quad
V^{3, 2}_{\text{HMZ}}=
\begin{pmatrix}
\mathcal S^{1,0,0}_{-1,-1,-1}\\
\mathcal S^{0,1,0}_{-1,-1,-1}\\
\mathcal S^{0,0,1}_{-1,-1,-1}
\end{pmatrix}.
\end{gather}
Thus, the polynomial spaces of their element correspond to the polynomial spaces in our construction with $p_1=p_2=p_3=2$, in which case we cannot use Hermitian degrees of freedom. Furthermore, we cannot fit the element into the full four row diagram~\eqref{diagram-4rows-V} as the space $V^{2, 2}_{\text{HMZ}}$ is not the range of a curl.
Accordingly, the regularity indices do not fit into our construction. The spaces for off-diagonal matrix entries exhibit more regularity than expected, but the divergence operator is onto, since the diagonal elements have the right regularity.
Hence, we conclude that pairs in the BGG complex can be constructed in a more general way than  our construction, but that this typically may not yield a whole FE complex with locally defined degrees of freedom.

\section{Conclusions}\label{sec:conclusion}

We presented a general construction of Bernstein-Gelfand-Gelfand (BGG) complexes by tensor products of piecewise polynomial spaces. 
The construction is based on merging two complexes of alternating forms of arbitrary degrees in $\mathbb{R}^n$ via cross-linking maps which commute with the differential operators of the complexes to obtain a single BGG complex. We first constructed the BGG complexes in one-dimension, then extended our construction to $n$-dimensions via the tensor product of spaces and the interpolation operators.

The method is based on scales of one-dimensional piecewise polynomial spaces characterized abstractly by their polynomial degree and smoothness parameter. Under the assumption of a commutativity between the interpolation operator of these spaces and the operators in the one-dimensional BGG diagram, we derived BGG complexes in any dimension with commuting interpolation operators.

We presented two examples for the application of this construction. For once, standard spline spaces in one dimension can be employed to generate the BGG spline complexes. Then, we presented the same construction for finite element spaces and showed their degrees of freedom in three dimensions.
Since the proposed construction relies on tensor products, it is naturally defined over cubical meshes with a straightforward extension to finite element meshes consisting of rectangular cells.
While this is clearly a limitation, it offers an important advancement in this context as it enables us to address problems in any spatial dimension using arbitrary polynomial degrees and smoothness on such meshes.

The construction for some more delicate BGG complexes remains open. We take the conformal deformation complex \cite[Equation (50)]{arnold2021complexes} as an example:
 \begin{equation}\label{conformal-deformation}
\begin{tikzcd}[row sep=small]
0 \arrow{r}&H^{q}\otimes \mathbb{V}\arrow{r}{\dev\deff} &H^{q-1}\otimes (\mathbb{S}\cap \mathbb{T}) \\& \arrow{r}{\cot} &H^{q-4}\otimes (\mathbb{S}\cap \mathbb{T}) \arrow{r}{\div} & H^{q-5}\otimes \mathbb{V} \arrow{r}{} & 0.
 \end{tikzcd}
\end{equation}
 Here $\dev\deff=\dev\sym\grad$ is the symmetric trace-free part of the gradient, and
 $\cot:={\curl  \mathcal{T}^{-1}\curl  \mathcal{T}^{-1}\curl}$ leads to the linearized Cotton-York tensor with modified trace. 
If the vector element on the left has tensor polynomial structure like
\begin{gather}
    \omega = \begin{pmatrix}
    p_1,q_1,r_1\\
    p_2,q_2,r_2\\
    p_3,q_3,r_3
    \end{pmatrix},
\end{gather}
in order to be symmetric and trace free, we look at the orders of the gradient
\begin{gather}
\nabla \omega =
    \begin{pmatrix}
    p_1-1,q_1,r_1 & p_1,q_1-1,r_1 & p_1,q_1,r_1-1\\
    p_2-1,q_2,r_2 & p_2,q_2-1,r_2 & p_2,q_2,r_2-1\\
    p_3-1,q_3,r_3 & p_3,q_3-1,r_3 & p_3,q_3,r_3-1
    \end{pmatrix}.
\end{gather}
For instance, in order to be symmetric, we must require $p_1 = p_2-1$, but for the trace adding up to zero, we need $p_2 = p_1-1$.
 
Based on the construction in this paper, one may further investigate numerical schemes in the direction of isogeometric analysis, c.f., \cite{arf2022mixed,evans2020hierarchical,kapidani2022high,zhang2018use}.

\appendix
\section{Proxy fields}
\label{sec:appendix-proxies}
When discussing concrete realizations of the diagram~\eqref{BGGdiagram-general} in spline and finite element spaces below, we will define them in terms of proxy fields for the differential forms involved. Thus, we need vector representations of alternating forms. In three dimensions, we can choose the basis $dx_1, dx_2, dx_3$ for $\alt^{0,1} \cong \alt^1$
and $\alt^{1,0} \cong \alt^1$. Thus, we obtain a basis for the according vector space by assigning $e_i=dx_i$ for $i=1,2,3$. For the spaces $\alt^{0,2}$ and $\alt^{2,0}$, we assign in similar fashion
\begin{gather}
    e_1 = dx_2\wedge dx_3,
    \quad e_2 = dx_3\wedge dx_1,
    \quad e_3 = dx_1\wedge dx_2.
\end{gather}
The space $\alt^{1,1} = \alt^1\otimes \alt^1$ consists of matrices spanned by the basis $e_{ij} = dx_i\otimes dx_j$ for $i,j=1,2,3$. A basis for $\alt^{1,2} = \alt^1\otimes \alt^2$ can be formed of elements
\begin{align}
  \label{tensorform-1-2}
    e_{i,1} = dx_i\otimes dx_2\wedge dx_3,
    \quad e_{i,2} = dx_i\otimes dx_3\wedge dx_1,
    \quad e_{i,3} =& dx_i\otimes dx_1\wedge dx_2,
 \\\nonumber &   \qquad i=1,2,3.
\end{align}

Following \cite{arnold2021complexes}, we introduce notation for the algebraic operations we need in this picture.  Now we consider the following  basic linear algebraic operations:    $\skw: \M\to \K$ and $\sym:\M\to\S$ are the skew and symmetric part operators; $\tr:\M\to\R$ is the matrix trace; $\iota: \R\to \M$ is the map $\iota u:= uI$ identifying a scalar with a scalar matrix; $\dev:\mathbb{M}\to \mathbb{T}$ given by $\dev w:=w-1/n \tr (w)I$ is the deviator, or trace-free part.
In three space dimensions, we can identify a skew symmetric matrix with a vector,
$$
 \mskw\left ( 
\begin{array}{c}
v_{1}\\ v_{2}\\ v_{3}
\end{array}
\right ):= \left ( 
\begin{array}{ccc}
0 & -v_{3} & v_{2} \\
v_{3} & 0 & -v_{1}\\
-v_{2} & v_{1} & 0
\end{array}
\right ).
$$

Consequently, we have $\mskw(v)w = v\times w$ for $v,w\in\V$,
where $v\times w$ denotes the cross product between $v=(v_1,v_2,v_3)$ and $w=(w_1,w_2,w_3)$ defined as
\[
v\times w=\left|
\begin{pmatrix}
e_1 & e_2 & e_3\\
v_1 & v_2 & v_3\\
w_1 & w_2 & w_3\\
\end{pmatrix}
\right|,
\]
$\{e_j\}_{j=1,2,3}$ denoting the standard basis of $\mathbb R^3$.
 Using the same notation as in 3D, we define $\mskw: \mathbb{R}\to \mathbb{K}$ in two-dimensions by 
$$
\mskw(u):= \left ( 
\begin{array}{cc}
0 &u\\
-u & 0
\end{array}
\right )\quad \mbox{in } \mathbb{R}^{2}.
$$
We also define the operators
\begin{xalignat*}2
    \vskw&\colon \M\to \V & \vskw&=\mskw^{-1}\circ \skw\\
    \opS&\colon \M\to \M&\opS u&=u^{T}-\tr(u)I.
\end{xalignat*}

\section{Characterization of $s$ operators} \label{sec:appendix}

In this appendix, we prove equation~\eqref{eq:s_iJ}, containing the expression of $s^{i, j}$ applied to a basis. We first recall,~\cite[Equation (2.1)]{Arnold.D;Falk.R;Winther.R.2006a}, which will be helpful below:
\begin{align}\label{eqn:form}
[dx^{1}\wedge \cdots\wedge& dx^{k}](u_{1}, \cdots, u_{k})\\&\nonumber
=\sum_{l=1}^{k}(-1)^{l+1}dx^{1}(u_{l})[dx^{2}\wedge\cdots\wedge dx^{k}](u_{1}, \cdots, \widehat{u_{l}}, \cdots, u_{k}).
\end{align}

\begin{lemma}\label{lem:s}
Let $\sigma\in\Sigma(k,n)$ and $\tau\in\Sigma(m,n)$. Then
\begin{gather}
    \label{eq:app-lemma-2}
    s^{k,m} \left(dx^\sigma\otimes dx^\tau\right)
    = \sum_{l=1}^m (-1)^{l-1} dx^{\tau_l} \wedge dx^\sigma
    \otimes dx^{\tau_1}\wedge \cdots \wedge \widehat{dx^{\tau_l}}\wedge\cdots\wedge dx^{\tau_m}.
\end{gather}

Let $\bm s\in \charset_k$ and $\bm t\in\charset_m$ be the characteristic vectors of $\sigma$ and $\tau$, respectively. Then, equation~\eqref{eq:app-lemma-2} can be equivalently written as
\begin{multline}
    \label{eq:app-lemma-1}
    s^{k, m}\bigl((dx^{1})^{s_1}\wedge\cdots \wedge (dx^{n})^{s_n}\otimes (dx^{1})^{t_1}\wedge \cdots \wedge (dx^{n})^{t_n}\bigr)\\
    =
    \sum_{l=1}^{n}(-1)^{|\bm{t}|_{l}+1}\delta_{1,t_l}dx^{l}\wedge (dx^1)^{s_1}\wedge\cdots \wedge (dx^n)^{s_n}\otimes \cdots\\
    (dx^1)^{t_1}\wedge \cdots \wedge \widehat{dx^l}\wedge\cdots\wedge (dx^n)^{t_n}.
\end{multline}
\end{lemma}

\begin{proof}
Let $\sigma\in\Sigma(k,n)$ and let $\tau\in\Sigma(m,n)$.
By the definition of $s^{k,m}$ in equation~\eqref{def:s}, there holds
\begin{align*}
T_1 &:=
    s^{k,m}\left[dx^\sigma\otimes dx^\tau\right]
    (e_{\br_1},\dots,e_{\br_{k+1}})
    (e_{\beta_1},\dots,e_{\beta_{m-1}})\\
    &=\sum_{\ell=1}^{k+1} (-1)^{\ell+1} \left[dx^\sigma\otimes dx^\tau\right]
    (e_{\br_1},\dots,\widehat{e_{\br_\ell}},\dots, e_{\br_{k+1}})
    (e_{\br_\ell},e_{\beta_1},\dots,e_{\beta_{m-1}})\\
    &=\sum_{\ell=1}^{k+1} (-1)^{\ell+1} 
    dx^\sigma (e_{\br_1},\dots,\widehat{e_{\br_\ell}},\dots, e_{\br_{k+1}})
    dx^\tau (e_{\br_\ell},e_{\beta_1},\dots,e_{\beta_{m-1}}),
\end{align*}
for $\br\in\Sigma(k+1,n)$ and $\beta\in\Sigma(m-1,n)$. We observe that at most one term in the sum on the right is nonzero, namely where $\ell$ can be chosen such there holds
\begin{gather}
    \{\br_1,\dots,\br_{k+1}\}
    = \{\br_\ell,\sigma_1,\dots,\sigma_k\},
    \qquad
    \{\br_\ell,\beta_1,\dots,\beta_{m-1}\}
    =\{\tau_1,\dots,\tau_m\}.
\end{gather}
In particular, we note that $\br_\ell\neq\sigma_i$
for $i=1,\dots,k$ and there is an index $j$ such that $\br_\ell = \tau_j$. Define $\tau'\in\Sigma(m-1,n)$ by removing $\tau_j$ from $\tau$. Then, by~\eqref{eqn:form},
\begin{align*}
    dx^\tau(e_{\br_\ell},e_{\beta_1},\dots,e_{\beta_{m-1}})
    &=
    (-1)^{j+1} [dx^{\tau_j}\wedge dx^{\tau'}](e_{\br_\ell},e_{\beta_1},\dots,e_{\beta_{m-1}})\\
    &= (-1)^{j+1} dx^{\tau_j}(e_{\br_\ell})
    dx^{\tau'}(e_{\beta_1},\dots,e_{\beta_{m-1}})\\
    &= (-1)^{j+1}
    dx^{\tau'}(e_{\beta_1},\dots,e_{\beta_{m-1}}).
\end{align*}
Thus,
\begin{gather}
    \label{eq:app-T1-final}
    T_1 = (-1)^{\ell+j} \delta_{\br_\ell,\tau_j}
    dx^\sigma(e_{\br_1},\dots,\widehat{e_{\br_\ell}},\dots, e_{\br_{k+1}})
    dx^{\tau'}(e_{\beta_1},\dots,e_{\beta_{m-1}}).
\end{gather}
Next, we start from~\eqref{eq:app-lemma-2} and let (again using $\tau'$ for $\tau$ without $\tau_j$)
\begin{align*}
    T_2 &:= 
    \sum_{j=1}^m (-1)^{j+1}
    \left[dx^{\tau_j} \wedge dx^\sigma \otimes dx^{\tau'}\right]
    (e_{\br_1},\dots,e_{\br_{k+1}})
    (e_{\beta_1},\dots,e_{\beta_{m-1}})\\
    &= \sum_{j=1}^m (-1)^{j+1}
    \left[dx^{\tau_j} \wedge dx^\sigma\right]
    (e_{\br_1},\dots,e_{\br_{k+1}})
    dx^{\tau'}
    (e_{\beta_1},\dots,e_{\beta_{m-1}}).
\end{align*}
The last sum contains at most one nonzero term, namely if $j$ can be chosen such that there holds
\begin{gather}
    \{\tau_j,\beta_1,\dots,\beta_{m-1}\}
    =\{\tau_1,\dots,\tau_m\},
    \qquad
    \{\br_1,\dots,\br_{k+1}\}
    =\{\tau_j,\sigma_1,\dots,\sigma_k\}.
\end{gather}
Since $\br_1,\dots,\br_{k+1}$ is as in~\eqref{eq:app-T1-final}, the index not contained in $\sigma$ is $\br_\ell$. By~\eqref{eqn:form} we get
\begin{gather}
    [dx^{\tau_j} \wedge dx^\sigma]
    (e_{\br_1},\dots,e_{\br_{k+1}})
    = (-1)^{\ell+1} dx^{\tau_j}(e_{\br_\ell})
    dx^\sigma(e_{\br_1},\dots,\widehat{e_{\br_\ell}},\dots,e_{\br_{k+1}}).
\end{gather}
Summarizing, we obtain
\begin{gather}
    \label{eq:app-T2-final}
    T_2 = (-1)^{\ell+j} \delta_{\br_\ell,\tau_j}
    dx^\sigma(e_{\br_1},\dots,\widehat{e_{\br_\ell}},\dots, e_{\br_{k+1}})
    dx^{\tau'}(e_{\beta_1},\dots,e_{\beta_{m-1}}),
\end{gather}
which is equal to $T_1$ in~\eqref{eq:app-T1-final}.
\end{proof}

We give a simple example to illustrate the lemma. For $k=m=1$, consider $s^{1, 1}[dx^1\otimes dx^2](e_1, e_2)()$, where $e_i$ is the dual basis of $dx^i$, $i=1, 2$. 
To better understand the notation, recall that $s^{1,1}$ maps $\alt^{1,1}$ in $\alt^{2,0}$. In particular, $s^{1, 1}[dx^1\otimes dx^2]$ is applied to $(e_1, e_2)$ in the first direction, whereas the second input is empty.
By definition of $s$, we have that the above is equal to
$$
[dx^1\otimes dx^2](e_2)(e_1)-[dx^1\otimes dx^2](e_1)(e_2)=-1.
$$
By Lemma \ref{lem:s}, the above is equal to 
$$
2\bigl[[dx^2\wedge dx^1]\otimes 1-[dx^1\wedge dx^2]\otimes 1\bigr](e_1, e_2)=2\cdot (0-\frac{1}{2})=-1.
$$

\section{On change of basis for interpolation operators}
\label{app:interpolation}
As a result of the peer review process, we are adding this appendix detailing the validity of  Remark~\ref{rem:change-of-basis}. We begin by proving that the standard and commuting interpolation operators yield the same result. We phrase this as the following general lemma:

\begin{lemma}
    Let $\nodal_1,\dots,\nodal_n$ be a set of node functionals on a suitable function space $V$ and let $\phi_1,\dots,\phi_n$ be a basis of an $n$-dimensional subspace $V_n\subset V$, such that the interpolation condition $\nodal_i(\phi_j) = \delta_{ij}$ holds for $i,j=1,\dots,n$.
    Let $\mathcal M_1,\dots,\mathcal M_n$ be a second set of node functionals on $V$ obtained by linearly independent linear combinations of the first set and let  $\psi_1,\dots,\psi_n$ be the basis of $V_n$ such that $\mathcal M_i(\psi_j) = \delta_{ij}$.
    
    Then, for any function $u$ for which $\nodal_1,\dots,\nodal_n$ are well-defined, there holds
    \begin{gather}
    \label{eq:appc-1}
        \sum_{i=1}^n \nodal_i(u)\phi_i
        =
        \sum_{i=1}^n \mathcal M_i(u)\psi_i.
    \end{gather}
    In particular, interpolation operators defined by the two sets of node values are identical.
\end{lemma}

\begin{proof}
    Let $u\in V$ arbitrary and let the vectors $v(u),w(u)\in\R^n$ be defined by $v_i(u) = \nodal_i(u)$ and $w_i(u) = \mathcal M_i(u)$, respectively.
    Due to the assumption on the two sets of node functionals, there is a nonsingular matrix $P$ such that $w(u) = P v(u)$ independent of $u$. 
    Let now for any $x$
    \begin{align*}
        \Phi(x) &= \bigl(\phi_1(x),\dots,\phi_n(x)\bigr)^T,\\
        \Psi(x) &= \bigl(\psi_1(x),\dots,\psi_n(x)\bigr)^T,
    \end{align*}
    For $\phi_j$, there holds
    \begin{align*}
        \phi_j
        &= \sum_i \mathcal M_i(\phi_j) \psi_i\\
        &= \sum_i\sum_k (P)_{ik} \nodal_k(\phi_j) \psi_i\\
        &= \sum_i (P)_{ij} \psi_i.
    \end{align*}
    Thus, $\Phi(x) = P^{T}\Psi(x)$ for any $x$.
    For any $u$, we conclude
    \begin{align*}
        \sum_i \nodal_i(u)\phi_i(x)
        &= v(u) \cdot \Phi(x)\\
        &= \bigl(P^{-1} w(u)\bigr) \cdot \bigl(P^T \Psi(x)\bigr)\\
        &= \bigl(P P^{-1} w(u)\bigr) \cdot \Psi(x)\\
        &= \sum_i \mathcal M_i(u)\psi_i(x).
    \end{align*}
\end{proof}

Introduce now on the whole mesh an interpolation operator using instead of~\eqref{eq:nodal00-2} and~\eqref{eq:nodal00-3} the standard node functionals $u(0)$ and $u(1)$.
By sharing these node functionals between adjacent cells, the range of this interpolation operator is the piecewise polynomial subspace conforming with the exterior derivative.
Hence, in the implementation, the standard interpolation operator involving $u(0)$ and $u(1)$ instead of~\eqref{eq:commuting-pair} as degrees of freedom is used.

In order to conduct the analysis, on each cell we use the preceding lemma to reinterpret the result as if it was obtained with the commuting degrees of freedom. Hence, the analysis applies.

\color{black}

\bibliographystyle{abbrv}      
\bibliography{reference}{}   

\end{document}